\documentclass{amsart}
\usepackage{amscd}
\usepackage{amsfonts}
\usepackage{amsmath}
\usepackage{amssymb}
\usepackage{amsthm}
\usepackage[english]{babel}
\usepackage[mathscr]{eucal}
\usepackage[T1]{fontenc}
\usepackage{indentfirst}
\usepackage[utf8]{inputenc}
\usepackage{graphics}
\usepackage{graphicx}
\usepackage[pagewise]{lineno}
\usepackage{mathrsfs}
\usepackage{pict2e}
\usepackage{t1enc}
\usepackage{tikz}
\usepackage{tikz-cd}
\usepackage{epic}
\numberwithin{equation}{section}
\usepackage{epstopdf}
\usepackage{setspace}
\usepackage{enumerate}

\newcommand{\cstar}{$\textrm{C}^*$}

\newtheorem{theo}{Theorem}[section]
\newtheorem{theoalph}{Theorem}
\newtheorem{cor}[theo]{Corollary}
\newtheorem{lem}[theo]{Lemma}
\newtheorem{prop}[theo]{Proposition}

\newtheorem*{claim}{Claim}

\theoremstyle{definition}

\newtheorem{defin}[theo]{Definition}

\newtheorem{quest}[theo]{Question}

\theoremstyle{remark}

\newtheorem{rem}[theo]{Remark}
     
\newcommand{\bbT}{\mathbb{T}}

\newcommand{\bC}{\mathbb{C}}
\newcommand{\bE}{\mathbb{E}}
\newcommand{\bN}{\mathbb{N}}
\newcommand{\bP}{\mathbb{P}}
\newcommand{\bQ}{\mathbb{Q}}
\newcommand{\bR}{\mathbb{R}}
\newcommand{\bZ}{\mathbb{Z}}
\newcommand{\cA}{\mathcal{A}}

\newcommand{\cD}{\mathcal{D}}

\newcommand{\kG}{\mathfrak{G}}
\newcommand{\kH}{\mathfrak{H}}
\newcommand{\kc}{2^{\aleph_0}}

\newcommand{\sB}{\mathscr{B}}

\newcommand{\sK}{\mathscr{K}}

\newcommand{\sM}{\mathscr{M}}

\newcommand{\aut}{\text{Aut}}

\newcommand{\C}{\textrm{C}^*}

\newcommand{\ax}[1]{\operatorname{\mathsf{#1}}}

\DeclareMathOperator{\ad}{Ad}

\newcommand{\cB}{\mathscr{B}}
\newcommand{\cK}{\mathscr{K}}
\newcommand{\bbC}{\mathbb{C}}
\newcommand{\bbZ}{\mathbb{Z}}
\newcommand{\bbR}{\mathbb{R}}
\newcommand{\bbN}{\mathbb{N}}
\newcommand{\CAR}{M_{2^\infty}}
\newcommand{\bbP}{\mathbb P}
\newcommand{\bbQ}{\mathbb Q}
\DeclareMathOperator{\cov}{cov} 
 
\newcommand{\cM}{\mathscr M}
\DeclareMathOperator{\Ad}{Ad}
\newcommand{\st}{\mathsf{S}}
\newcommand{\pure}{\mathsf{P}}
\newcommand{\U}{\mathsf{U}}
\newcommand{\diacoh}{\diamondsuit^\mathsf{Cohen}}
\newcommand{\rs}{\upharpoonright}
\DeclareMathOperator{\id}{id}

\newcommand{\e}{\varepsilon}

\newcommand{\bEA}[2]{\bE_{A^\circ}(#1, #2)}
\newcommand{\bEB}[2]{\bE_{B^\circ}(#1, #2)}

\newcommand{\sfF}{\mathsf F}
\DeclareMathOperator{\Aut}{Aut}
\newcommand{\sfC}{\mathsf C}
\newcommand{\ZFCmP}{$\ax{ZFC-P}$}
\newcommand{\ZFC}{$\ax{ZFC}$}

\begin{document}

\title{Can you take Akemann--Weaver's $\diamondsuit_{\aleph_1}$ away?}

\author[Daniel Calder\'on]{Daniel Calder\'on}
\address{Department of Mathematics and Statistics, York University, 4700
Keele Street, Toronto, Ontario, Canada, M3J 1P3}
\email{d.calderon@mail.utoronto.ca}

\author[Ilijas Farah]{Ilijas Farah}
\thanks{Partially supported by NSERC}
\address{Department of Mathematics and Statistics, York University, 4700
Keele Street, Toronto, Ontario, Canada, M3J 1P3} 
\address{Matemati\v{c}ki Institut SANU, Kneza Mihaila 36, 11000 Beograd, p.p.\,367, Serbia}
\email{ifarah@yorku.ca}

\date{\today}

\subjclass[2010]{Primary: 03E35, 46L30.}

\keywords{Representations of \cstar-algebras, forcing, Jensen's Diamond, Naimark's problem} 

\thanks{Corresponding author: Ilijas Farah. ORCID: 0000-0001-7703-6931}

\maketitle

{\centering\footnotesize\emph{A Carlos Di Prisco en su cumplea\~nos n\'umero 70.}\par}

\begin{abstract}
By Glimm's dichotomy, a separable, simple \cstar-algebra has continuum many unitarily inequivalent irreducible representations if, and only if, it is non-type I while all of its irreducible representations are unitarily equivalent if, and only if, it is type I. Naimark asked whether the latter equivalence holds for all \cstar-algebras. 

In 2004, Akemann and Weaver gave a negative answer to Naimark's problem using Jensen's Diamond Principle $\diamondsuit_{\aleph_1}$, a powerful diagonalization principle that implies the Continuum Hypothesis ($\ax{CH}$). By a result of Rosenberg, a separably represented, simple \cstar-algebra with a unique irreducible representation is necessarily of type I. We show that this result is sharp by constructing an example of a separably represented, simple \cstar-algebra that has exactly two inequivalent irreducible representations, and therefore does not satisfy the conclusion of Glimm's dichotomy. Our construction uses a weakening of Jensen's $\diamondsuit_{\aleph_1}$, denoted $\diacoh$, that holds in the original Cohen's model for the negation of $\ax{CH}$. We also prove that $\diacoh$ suffices to give a negative answer to Naimark's problem. Our main technical tool is a forcing notion that generically adds an automorphism of a given \cstar-algebra with a prescribed action on its space of pure states. 
\end{abstract}

\maketitle

\section{Introduction}

A major early result in the theory of operator algebras (and, at the time, possibly the deepest result in the theory; see \cite[\S IV.1.5]{Black:Operator}) was Glimm's 1960 dichotomy theorem. It states (among other things) that a separable and simple \cstar-algebra $A$ either has a unique irreducible representation up to unitary equivalence, or it has $2^{\aleph_0}$ inequivalent irreducible representations. The former condition is equivalent to $A$ being isomorphic to the algebra of compact operators on a separable Hilbert space, while the latter is equivalent to $A$ not being of type I (see \cite[Theorem~IV.1.5.1]{Black:Operator} for the full statement).  

Parts of Glimm's theorem were extended to non-separable \cstar-algebras by Sakai (see \cite[IV.1.5.8]{Black:Operator} for a discussion). In the 1970s further progress on extending Glimm's theorem to all simple \cstar-algebras slowed down to a halt. The most obvious question, asked by Naimark already in the 1950s, was whether a \cstar-algebra with a unique irreducible representation up to unitary equivalence is necessarily isomorphic to an algebra of compact operators. A \emph{counterexample to Naimark's problem} is a \cstar-algebra that is not isomorphic to an algebra of compact operators, yet still has only one irreducible representation up to unitary equivalence.

In a seminal paper \cite{AkeWe:Consistency}, Akemann and Weaver constructed a counterexample to Naimark's problem using Jensen's $\diamondsuit_{\aleph_1}$ principle\footnote{It is not known whether $\diamondsuit_\kappa$ has any bearing on Naimark's problem for any $\kappa\geq\aleph_2$.}. By related proofs, also conditioned on Jensen's $\diamondsuit_{\aleph_1}$ on $\aleph_1$, several counterexamples with additional properties (e.g., a prescribed tracial simplex \cite{vaccaro2017trace}, not isomorphic to its opposite algebra \cite{farah2016simple}) were obtained. N.C. Phillips observed that the Akemann--Weaver construction provides a nuclear \cstar-algebra. The range of applications of this construction was extended to other problems: In \cite{farah2016simple} it was shown that Glimm's dichotomy can fail: Assuming $\diamondsuit_{\aleph_1}$, there exists a simple \cstar-algebra with exactly $m$ inequivalent irreducible representations for all $m\leq\aleph_0$ (the latter was announced in \cite[\S 8.2]{Fa:Logic}) and a hyperfinite II$_1$ factor not isomorphic to its opposite was constructed in \cite{farah2021rigid}. 	 

The following is a special case of our Theorem \ref{realmain}. 

\begin{theoalph}\label{main}
It is relatively consistent with $\ax{ZFC}$ that there exists a counterexample to Naimark's problem while $\diamondsuit_{\aleph_1}$ fails.
\end{theoalph}

More specifically, we isolate a combinatorial principle $\diacoh$ that, together with the Continuum Hypothesis, implies the existence of a counterexample to Naimark's problem. Then we show that $\diacoh+\ax{CH}$ is consistent with the failure of $\diamondsuit_{\aleph_1}$ (see Appendix \ref{S.diamond}). 

A simple \cstar-algebra with a unique irreducible representation up to unitary equivalence that is represented on a Hilbert space of density character strictly smaller than $2^{\aleph_0}$ is necessarily isomorphic to an algebra of compact operators (see e.g., \cite[Corollary~5.5.6]{Fa:STCstar}). We prove the `next  best thing' by constructing a separably represented counterexample to Glimm's dichotomy. More precisely, we obtain the following: 

\begin{theoalph}\label{C.separable}
For any $m\geq2$, $\diacoh+\ax{CH}$ implies that there exists a separably represented, simple, unital \cstar-algebra with exactly $m$ irreducible representations up to unitary equivalence.
\end{theoalph}

This is a special case of Theorem \ref{T.separable} proved below. 

Our principal technical contribution is the introduction of a forcing notion that generically adds an automorphism of a given \cstar-algebra with a prescribed action on its space of pure states (see \S\ref{S.forcingE}, \S\ref{S.unique}, and \S\ref{S.X}), which constitute a generalization of a celebrated theorem of Kishimoto, Ozawa, and Sakai (see \cite{KiOzSa}).

\subsubsection*{Acknowledgments}

Some of the results of this paper come from the first author's masters thesis written under the second author's supervision. We are indebted to Ryszard Nest and Chris Schafhauser for enlightening discussions. We also wish to thank Andrea Vaccaro and Alessandro Vignati for helpful comments on the early versions of this paper, and to  Assaf Rinot for precious comments on precious stones\footnote{Diamonds.} that considerably improved the presentation of the transfinite constructions in this paper.

\section{Preliminaries and notation}

The reader is assumed to be familiar with the basics of forcing and the basics of \cstar-algebras. Our notation and terminology follow \cite{Black:Operator} for operator algebras, \cite{Jech:SetTheory} and \cite{Ku:Set} for set theory (in particular, forcing), and \cite{Fa:STCstar} for both (except forcing). It is understood that \cite{Fa:STCstar} is used as a reference only for the reader's (as well as the author's) convenience; none of the results referred to are claimed to be due to the author of \cite{Fa:STCstar}.

\subsection{\cstar-algebras and their representations}

\cstar-algebras are complex Banach algebras with an involution $^*$ that satisfy the \cstar-equality, $\|aa^*\|=\|a\|^2$. By a result of Gelfand and Naimark, these are exactly the algebras isomorphic to a norm-closed, $^*$-closed, subalgebra of the algebra $\cB(H)$ of bounded linear operators on a complex Hilbert space $H$. A homomorphism between \cstar-algebras that preserves the involution is called a \emph{$^*$-homomorphism}, and a $^*$-homomorphism into $\cB(H)$ is a \emph{representation}. A representation $\pi\colon A\to\cB(H)$ is \emph{irreducible} if $H$ has no nontrivial closed subspaces invariant under the image of $A$. A representation is \emph{faithful} if it is injective. Every faithful representation is necessarily isometric. More generally, all $^*$-homomorphisms are contractive (take note that in the theory of operator algebras `contractive' is synonymous with `1-Lipshitz.') When $A$ is unital, $\U(A)$ is the set of its unitary elements, i.e., those $u\in A$ such that $uu^*=u^*u=1_A$.

\subsection{States}\label{subsecstates}

A bounded linear functional $\varphi$ on a \cstar-algebra $A$ is a \emph{state} if it is \emph{positive}, i.e., $\varphi(a^*a)\geq 0$ for all $a\in A$, and $\|\varphi\|=1$. The space of all states on~$A$ is denoted $\st(A)$. Via the Gelfand--Naimark--Segal (GNS) construction (see \cite[\S 1.10]{Fa:STCstar}), every state $\varphi$ on $A$ is associated with a representation $\pi_\varphi\colon A\to\cB(H_\varphi)$ such that a unique (up to multiplication by a scalar of modulus 1)  unit vector,  $\xi_\varphi$ in $H_\varphi$ satisfies that $\varphi(a)=(\pi_\varphi(a)\xi_\varphi|\xi_\varphi)$ for all $a\in A$. The triplet $(\pi_\varphi,H_\varphi,\xi_\varphi)$ is the \emph{GNS triplet} associated with $\varphi$. Conversely, every representation $(\pi,H)$ of $A$ with a cyclic vector (i.e., some $\xi\in H$ such that the $\pi[A]$-orbit of $\xi$ is dense in $H$) is of the form $\pi_\varphi$ for some state $\varphi$ on $\st(A)$. When $A$ is unital, the states space of $A$ is a weak$^*$-compact and convex set. The extreme points of $\st(A)$ are called \emph{pure states} and a state is pure if, and only if, the corresponding GNS representation is \emph{irreducible}, if, and only if, $\pi_\varphi[A]$ is dense in $\cB(H_\varphi)$ with respect to the weak operator topology (see \cite[\S 3.6]{Fa:STCstar}). The space of all pure states on $A$ is denoted~$\pure(A)$.

\subsection{Automorphisms and crossed products}

We will say that an automorphism $\Phi$ of a \cstar-algebra $A$ is \emph{inner} if it is of the form $\ad u(a):=uau^*$ for some unitary $u$ in the unitization of $A$. It is \emph{approximately inner} if there exists a net of unitaries $(u_p:p\in\kG)$ in the unitization of $A$ such that $\Phi(a)=\lim_\kG\ad u_p(a)$ for each $a\in A$. The automorphism group of $A$ is denoted $\Aut(A)$. 

An automorphism $\Phi$ of $A$ determines a continuous action of $\bZ$ on $A$ given by $n.a:=\Phi^n(a)$. To such a non-commutative dynamical system one can associate a \emph{reduced crossed product}, $A\rtimes_\Phi\bZ$. This \cstar-algebra is generated by an isomorphic copy of $A$ (routinely identified with $A$) and a unitary $u$ that implements $\Phi$ on $A$, in the sense that $\Ad u(a)=\Phi(a)$ for all $a\in A$. For more details see e.g., \cite[\S 4.1]{BrOz:C*} or \cite[\S 2.4.2]{Fa:STCstar}.

\subsection{Equivalences of states and representations}\label{S.equiv}

We say that two representations $(\pi_0,H_0)$ and $(\pi_1,H_1)$ of a \cstar-algebra are \emph{spatially equivalent}, and we write $\pi_0\sim\pi_1$, if there exists a *-isomorphism $\Phi\colon\sB(H_0)\to\sB(H_1)$ such that $\Phi\circ\pi_0=\pi_1$. Two pure states $\varphi$ and $\psi$ of $A$ are called \emph{conjugate} if there exists an automorphism $\Phi$ of $A$ such that $\varphi\circ\Phi=\psi$. If $\Phi$ can be chosen to be inner, we say that $\varphi$ and $\psi$ are \emph{unitarily equivalent} and write $\varphi\sim\psi$. Using a GNS argument plus the Kadison transitivity theorem, it can be shown (see \cite[Lemma 3.8.1]{Fa:STCstar}) that $\varphi\sim\psi$ if, and only if, $\pi_\varphi\sim\pi_\psi$, if, and only if, there is a unitary $u$ in the unitization of $A$ such that $\|\varphi\circ\ad u-\psi\|<2$. This implies that Naimark's problem as stated in the previous section is equivalent to asking whether every \cstar-algebra with a unique pure state up to unitary equivalence is isomorphic to an algebra of compact operators.

\subsection{The space $\pure_m(A)$}\label{S.pure}

Following \cite[\S 5.6]{Fa:STCstar}, for $m\in\bN$, the space of $m$-tuples of pairwise inequivalent pure states of $A$ is denoted $\pure_m(A)$. A typical element of $\pure_m(A)$ is of the form $\bar\varphi=(\varphi_i:i<m)$. The automorphism group of $A$ acts naturally on $\pure_m(A)$ as follows: If $\bar{\varphi}\in\pure_m(A)$ and $\Phi\in\aut(A)$ then $\bar{\varphi}\circ\Phi:=(\varphi_i\circ\Phi:i<m)$.

We write $G\Subset A$ if $G$ is a finite subset of $A$. If $G\Subset A$ and $\delta>0$, we write
\begin{equation*}
    \bar{\varphi}\approx_{G,\delta}\bar{\psi}\quad\text{if, and only if,}\quad\max_{b\in G}\left(\max_{i<m}|\varphi_i(b)-\psi_i(b)|\right)<\delta.
\end{equation*}
Thus, 
\[
U_{G,\delta}(\bar \varphi):=\{\bar\psi:\bar\psi\approx_{G,\delta} \bar\varphi\}
\]
 is a typical weak$^*$-open neighbourhood of $\bar\varphi$ in $\pure_m(A)$. These sets range over a weak$^*$-neighbourhood basis of $\bar\varphi$ in $\pure_m(A)$ as $\delta$ ranges over positive reals and $G$ ranges over finite subsets of $A$ (or finite subsets of a fixed dense subset of $A$, for this use the fact that states have norm $1$).  

It is worth mentioning that the notation $\bar\varphi\sim\bar\psi$ is reserved for the existence of a unitary $u\in\U(A)$ such that $\bar\varphi\circ\ad u=\bar\psi$ (cf. Definition \ref{pwue}).

Given two tuples of pairwise inequivalent pure states $\bar\varphi\in\pure_m(A)$ and $\bar\psi\in\pure_l(A)$, we denote by $\bar\varphi^\frown\bar\psi$ their \emph{concatenation} $(\varphi_0,\dots,\varphi_{m-1},\psi_0,\dots,\psi_{l-1})\in\pure(A)^{m+l}$.

\subsection{Type I \cstar-algebras}

A \cstar-algebra $A$ is \emph{type I} if the ideal of compact operators on $H$ is included in $\pi[A]$ for every irreducible representation $(\pi,H)$ of $A$, and \emph{non-type I} if it is not type I. An example of a non-type I \cstar-algebra is the \emph{CAR algebra}, $\CAR:=\bigotimes_\bbN M_2(\bC)$, and by a result due to Glimm (see \cite[Theorem 3.7.2]{Fa:STCstar}), a \cstar-algebra is non-type I if, and only if, it has a $\C$-subalgebra which has a quotient isomorphic to $\CAR$.

\subsection{Transitive models of \ZFCmP}\label{S.models}

Our ambient theory is $\ax{ZFC}$, the Zermelo--Fraenkel set theory with the Axiom of Choice (see e.g., \cite[\S A.1]{Fa:STCstar}). Because of meta-mathematical obstructions of no direct relevance to the present paper, while working in $\ax{ZFC}$ one cannot prove the existence of a model of $\ax{ZFC}$. Fortunately, for any uncountable regular cardinal $\kappa$ the set $H(\kappa)$ of all sets whose hereditary closure has cardinality smaller than $\kappa$ (see e.g., \cite[\S A.7]{Fa:STCstar}) is a model of \ZFCmP, the theory obtained by removing the Power Set axiom from $\ax{ZFC}$.\footnote{Purists may prefer working with transitive structures closed under the rudimentary functions, see e.g., \cite[Definition 27.2]{Jech:SetTheory}.} This fragment of $\ax{ZFC}$ suffices for our purposes.\footnote{It may be worth mentioning that the Power Set axiom is far from being useless. By a result of Harvey Friedman, Borel Determinacy cannot be proved in \ZFCmP--its proof even requires uncountably many iterations of the power set operation--and yet it is a theorem of \ZFC; see \cite{martin1975borel}.} 

Borel subsets of a Polish space with a fixed countable basis can be coded by elements of $\bbN^\bbN$ (see e.g., \cite[p. 504]{Jech:SetTheory}). This coding is sufficiently absolute, so that a transitive model of \ZFCmP{} that does not include the set of all real numbers can still contains codes for some Borel sets and be correct about their properties (the proof of \cite[Lemma~25.46]{Jech:SetTheory} applies to show this).

\subsection{Forcing}\label{S.forcing}

A \emph{forcing notion} is a partially ordered set $\bbP$. The elements of $\bbP$ are also called \emph{conditions}, and if $p\leq q$ then $p$ is said to \emph{extend} $q$. Two conditions are \emph{compatible} if a single condition extends both of them. A subset $D$ of $\bbP$ is called \emph{open} if it contains all extensions of all of its elements. A subset $D$ of $\bbP$ is called \emph{dense} if it contains some extension of every condition in $\bbP$. A subset $\kG$ of $\bbP$ is a \emph{filter} if it satisfies the following two conditions: (i) $p\in \kG$ and $p\leq q$ implies $q\in \kG$, and (ii) every two elements of $\kG$ have a common extension in $\kG$. If $\cD$ is a family of dense open subsets of $\bbP$, then a filter $\kG$ is called \emph{$\cD$-generic} if it intersects every element of $\cD$ non-trivially.  

If $M$ is a transitive model of \ZFCmP, $\bbP$ is a forcing notion in $M$, and a filter $\kG\subseteq \bbP$ intersects all dense open subsets of $\bbP$ that belong to $M$, then $\kG$ is said to be \emph{$M$-generic}. In this situation, one can define the forcing (also called generic) extension $M[\kG]$ which is a transitive model of \ZFCmP{} that includes $M$ and contains $\kG$. The model $M$ is usually referred to as the \emph{ground model}.

\section{Forcing an approximately inner automorphism}\label{S.forcingE}

Let $A$ be a simple and unital \cstar-algebra. Given two elements $\bar\varphi$ and $\bar\psi$ of $\pure_m(A)$, we will define a forcing notion $\bE_{A^\circ}(\bar\varphi,\bar\psi)$, depending on a $\bQ+i\bQ$-subalgebra $A^\circ$ of $A$, whose generic object codes an approximately inner automorphism $\Phi_\kG$ of $A$ such that $\bar\varphi\circ\Phi_\kG=\bar\psi$.

Besides $\diamondsuit_{\aleph_1}$, the Akemann--Weaver construction uses a refinement of a deep 2001 result due to Kishimoto, Ozawa, and Sakai (see \cite{KiOzSa}, also \cite[\S 5.6]{Fa:STCstar}) that implies that all pure states of a separable, simple, and unital $\C$-algebra are conjugate by an approximately (and even asymptotically) inner automorphism. A crucial lemma in the proof of the Kishimoto--Ozawa--Sakai theorem (see \cite[Lemma 2.2]{KiOzSa}, also \cite[Lemma~5.6.7]{Fa:STCstar}) together with \cite[Property 7]{FuKaKi} serves as a motivation for the following lemma. 

\begin{lem}\label{goodlemma}
Let $A$ be a simple, unital, and infinite-dimensional \cstar-algebra. For all $m\geq 1$, $\bar\varphi\in \pure_m(A)$, $F\Subset A$, and $\e>0$ the following holds: There exist a $G\Subset A$ and $\delta>0$ such that for all $\bar\theta\in\pure_m(A)$, if $\bar{\varphi}\approx_{G,\delta}\bar{\theta}$ then for all $K\Subset A$ and every $\gamma>0$ there exists a unitary $v\in\U(A)$ such that $\bar\varphi\circ\Ad v\approx_{K,\gamma}\bar\theta$ and $\|b-\Ad v(b)\|<\varepsilon$ for every $b\in F$.
\end{lem}

\begin{proof} We commence the proof by restating it in the language introduced in \S\ref{S.pure}. 
It asserts that for every $m\geq 1$,  all $F\Subset A$ and $\varepsilon>0$, and all $\bar \varphi\in \pure_m(A)$,  there is a basic open neighbourhood $U_{G,\delta}(\bar \varphi)$ of $\bar\varphi$ such that for every $\bar\theta$ in this neighbourhood and every basic open neighbourhood $U_{K,\gamma}(\bar \theta)$ of $\bar \theta$, some $v\in \U(A)$ satisfies $\max_{b\in F} \|[v,b]\|<\varepsilon$ and  $\bar\varphi\circ \Ad v\in U_{K,\gamma}(\bar \theta)$. 

Equivalently,  for every $m\geq 1$,  all $F\Subset A$ and $\varepsilon>0$, and all $\bar \varphi\in \pure_m(A)$, there is a basic open neighbourhood $U_{G,\delta}(\bar \varphi)$ of $\bar\varphi$ such that 
 the set 
 \[
 \{\bar\varphi\circ \Ad v\mid v\in\U(A), \max_{b\in F} \|[v,b]\|<\varepsilon\}
 \]
  is weak$^*$-dense in  $U_{G,\delta}(\bar \varphi)$. 

\cite[Lemma 5.6.7]{Fa:STCstar} falls just a little short of proving this, under the same assumptions on $A$. In the case when $A$ is unital, hence $A=\tilde A$, it asserts that for every $m\geq 1$, all $F\Subset A$ and $\varepsilon>0$, and all $\bar \varphi\in \pure_m(A)$, there is a basic open neighbourhood $U_{G,\delta'}(\bar \varphi)$ of $\bar\varphi$ such that for every $\bar \psi\in U_{G,\delta'}(\bar \varphi)$ with $\bar\psi\sim\bar\varphi$ (see \S\ref{S.equiv}), there exists $u_1\in \U(A)$ such that $\max_{b\in F} \|[b,u_1]\|<24\varepsilon$ and $\bar\varphi\circ \Ad u_1=\bar\psi$.\footnote{The unitary $u_1$ obtained in \cite[Lemma 5.6.7]{Fa:STCstar} is homotopic to $1_A$, and all the unitaries in the homotopy path satisfy $\max_{b\in F} \|[b,u_t]\|<24\varepsilon$, but we don't need this.}

Hence, in order to complete the proof, it will suffice to replace $\varepsilon$ with $\varepsilon/24$ and show that the set $\{\bar\psi\mid \bar\psi\sim\bar \varphi\}$ is dense in $U_{G,\delta'}$.   
Since $A$ is simple and infinite-dimensional, every representation $\pi\colon A\to \cB(H)$ satisfies $\pi[A]\cap \cK(H)=\{0\}$. It is a standard consequence of this condition and  Glimm's Lemma that  for every $\bar\psi$ in $\pure_m(A)$, the unitary orbit $\{\bar\psi\circ\ad u: u\in\U(A)\}$  is dense in $\pure_m(A)$  (see e.g., \cite[Proposition 5.2.9]{Fa:STCstar}). This completes the proof. 
\end{proof}

The previous lemma motivates the following definition.

\begin{defin}\label{D.1}
Given a \cstar-algebra $A$ and $m\geq 1$, let $\bar\varphi\in \pure_m(A)$, $F\Subset A$, and $\e>0$. We will say that a pair $(G,\delta)$ with $G\Subset A$ and $\delta>0$ is \emph{$(\bar\varphi,F,\varepsilon)$-good} if for all $\bar\theta\in\pure_m(A)$, if $\bar{\varphi}\approx_{G,\delta}\bar{\theta}$ then for all $K\Subset A$ and every $\gamma>0$ there exists a unitary $u\in\U(A)$ such that $\bar\varphi\circ\Ad u\approx_{K,\gamma}\bar\theta$ and $\|b-\Ad u(b)\|<\varepsilon$ for every $b\in F$.
\end{defin}

Analogously to the restatement of Lemma~\ref{goodlemma} given in the first paragraph of its proof, one obtains  the following. 

\begin{lem} For  $m\geq 1$,   $F\Subset A$,   $\varepsilon>0$, and  $\bar \varphi\in \pure_m(A)$,  a pair $(G,\delta)$ with $G\Subset A$ and $\delta>0$ is $(\bar\varphi,F,\varepsilon)$-good if and only if 
 the set 
 \[
 \{\bar\varphi\circ \Ad v\mid v\in\U(A), \max_{b\in F} \|[v,b]\|<\varepsilon\}
 \]
  is weak$^*$-dense in  $U_{G,\delta}(\bar \varphi)$. \qed
\end{lem} 

Thus we have the following equivalent reformulation of Lemma~\ref{goodlemma}. 
\begin{lem} 
 If  $A$ is a simple, unital,  infinite-dimensional \cstar-algebra, then for all $m\geq 1$, $\bar\varphi\in \pure_m(A)$, $F\Subset A$, and $\e>0$, there exists a $(\bar\varphi,F,\varepsilon)$-good pair $(G,\delta)$. \qed
\end{lem}

Fix now a simple and unital \cstar-algebra $A$, and tuples $\bar\varphi$ and $\bar\psi$ in $\pure_m(A)$. We also fix a norm-dense $\bQ+i\bQ$-subalgebra $A^\circ$ of $A$ with minimal cardinality such that $\U(A)\cap A^\circ$ is norm-dense in $\U(A)$. In particular, when $A$ is separable, $A^\circ$ will be countable.

In Definition \ref{poset}, we will introduce a forcing notion $\bE_{A^\circ}(\bar\varphi,\bar\psi)$ which generically adds an automorphism $\Phi_\kG$ of $A$ such that $\bar\varphi\circ\Phi_\kG=\bar\psi$. More precisely, it adds two nets of unitaries, $v_p$ and $w_p$, for $p\in \kG$, such that each of the nets of $\Ad v_p$ and $\Ad w_p^*$ indexed by $p\in \kG$ converges pointwise to an automorphism of $A$ (this is assured by condition (c)). The automorphism $\Phi_\kG$ is the composition of these two automorphisms. 

\begin{defin}\label{poset}
Let $\bE_{A^\circ}(\bar\varphi,\bar\psi)$ be the set of tuples
\begin{equation*}
q=(F_q,G_q,\varepsilon_q,\delta_q,v_q,w_q)    
\end{equation*}
such that:
\begin{enumerate}
    \item $F_q$ and $G_q$ are finite subsets of $A^\circ$,
    \item $\varepsilon_q$ and $\delta_q$ are positive real numbers,
    \item $v_q$ and $w_q$ are unitaries of $A$ in $A^\circ$,
    \item $(G_q,\delta_q)$ is a $(\bar\varphi\circ\ad v_q,F_q\cup\ad v_q^*[F_q],\varepsilon_q/3)$-good pair, and
    \item $\bar\varphi\circ\ad v_q\approx_{G_q,\delta_q}\bar\psi\circ\ad w_q$.
\end{enumerate}
We order $\bE_{A^\circ}(\bar\varphi,\bar\psi)$ by $p\leq q$ if:
\begin{enumerate}
    \item[(a)] $F_p\supseteq F_q$, $G_p\supseteq G_q$,
    \item[(b)] $\varepsilon_p\leq\varepsilon_q$, $\delta_p\leq\delta_q$, and
    \item[(c)] for all $b\in F_q$
    \begin{align*}
        \max\left\{\left\|\ad v_p(b)-\ad v_q(b)\right\|,\left\|\ad v_p^*(b)-\ad v_q^*(b)\right\|\right\}&\leq\varepsilon_q-\varepsilon_p,\text{ and}\\
        \max\left\{\left\|\ad w_p(b)-\ad w_q(b)\right\|,\left\|\ad w_p^*(b)-\ad w_q^*(b)\right\|\right\}&\leq\varepsilon_q-\varepsilon_p.
    \end{align*}
\end{enumerate}
\end{defin}

\begin{rem}
The bound $\varepsilon_q-\varepsilon_p$ in (c) of Definition \ref{poset} is used to assure that the relation $\leq$ is transitive on $\bE_{A^\circ}(\bar\varphi,\bar\psi)$. This idea is taken from  \cite{farahembedding}.
\end{rem}

\begin{lem}\label{density}
For all finite subsets $F$ and $G$ of $A^\circ$, and all positive real numbers $\varepsilon$ and $\delta$, the set $D(F,G,\varepsilon,\delta)$ of conditions $p$ such that $F\subseteq F_p$, $G\subseteq G_p$, $\varepsilon_p\leq\varepsilon$, and $\delta_p\leq\delta$ is a dense and open\footnote{See \S\ref{S.forcing} for definitions.} subset of $\bE_{A^\circ}(\bar\varphi,\bar\psi)$.
\end{lem}

\begin{proof}
It is obvious that  an extension of every condition in $D(F,G,\varepsilon,\delta)$ belongs to $D(F,G,\varepsilon,\delta)$.

Second, we need to prove that every condition has an extension that belongs to $D(F,G,\varepsilon,\delta)$. Let $q\in\bE_{A^\circ}(\bar\varphi,\bar\psi)$ be fixed but arbitrary. By Lemma \ref{goodlemma}, there exists a $(\bar\psi\circ\ad w_q,F_q\cup\ad w_q^*[F_q],\varepsilon_q)$-good pair; denote it $(G_1,\delta_1)$. Since $q$ is a condition, $(G_q,\delta_q)$ is $(\bar\varphi\circ\ad v_q,F_q\cup\ad v_q^*[F_q],\varepsilon_q)$-good and $\bar\varphi\circ\ad v_q\approx_{G_q,\delta_q}\bar\psi\circ\ad w_q$.

Using the goodness of $(G_q,\delta_q)$, choose $v\in\U(A)\cap A^\circ$ to be some unitary such that $\bar\varphi\circ\ad v_qv\approx_{G_1,\delta_1}\bar\psi\circ\ad w_q$ and $\|b-\ad v(b)\|<\varepsilon_q/3$ for all $b\in F_q\cup\ad v_q^*[F_q]$. Set $\varepsilon_p:=\min\{\varepsilon,\varepsilon_q/3\}$, $F_p:=F_q\cup F$, and $v_p:=v_qv$.

By Lemma \ref{goodlemma}, there is a $(\bar\varphi\circ\ad v_p,F_p\cup\ad v_p^*[F_p],\varepsilon_p)$-good pair, denoted $(G_2,\delta_2)$. Let $G_p:=G\cup G_q\cup G_2$ and $\delta_p:=\min\{\delta,\delta_q,\delta_2\}$.

Using now the fact that the pair $(G_1,\delta_1)$ is a good pair, let $w\in\U(A)\cap A^\circ$ be such that $\bar\varphi\circ\ad v_p\approx_{G_p,\delta_p}\bar\psi\circ\ad w_qw$ and $\|b-\ad w(b)\|<\varepsilon_q/3$ for every $b\in F_q\cup\ad w_q^*[F_q]$. Define $w_p$ as $w_qw$ and set $p$ to be $(F_p,G_p,\varepsilon_p,\delta_p,v_p,w_p)$.

Clearly $F\subseteq F_p$, $G\subseteq G_p$, $\varepsilon_p\leq\varepsilon$ and $\delta_p\leq\delta$. Also, since the pair $(G_p,\delta_p)$ is $(\bar\varphi\circ\ad v_p,F_p\cup\ad v_p^*[F_p],\varepsilon_p)$-good, $p\in\bE_{A^\circ}(\bar\varphi,\bar\psi)$. Finally, if $b\in F_q$ then
\begin{align*}
\|\ad v_q^*(b)-\ad v_p^*(b)\|&=\|\ad v(\ad v_q^*(b))-\ad v_q^*(b)\|\\
&<\varepsilon_q/3<\varepsilon_q-\varepsilon_q/3\leq\varepsilon_q-\varepsilon_p.    
\end{align*}
Also, $\|\ad v_q(b)-\ad v_p(b)\|=\|b-\ad v(b)\|<\varepsilon_q-\varepsilon_p$. The calculations for $w_p$ and $w_q$ are analogous and therefore $p\leq q$.
\end{proof}

Two forcing notions $\bP_0$ and $\bP_1$ are said to be  \emph{forcing-equivalent} if they give rise to the same generic extensions, i.e., every generic extension $M[G]$ by $\bP_0$ can be realized as a generic extension $M[H]$ for some $M$-generic filter $H$ on $\bP_1$, and vice versa. The simplest mechanism for assuring forcing equivalence is by finding an isomorphic copy of $\bP_0$ in $\bP_1$ that is also dense (in the forcing sense, see \S\ref{S.forcing}) in it. (A proof of this fact is contained in \cite[{Lemma~III.3.68 ($\beta$)}]{Ku:Set}. The only point is that if $\bbP$ is a dense subordering of $\bbP_1$ then a filter $\kG\subseteq \bbP_1$ is generic over $M$ if and only if $\kG\cap \bbP$ is generic over $M$.) It is straightforward to use this observation to prove that every two countable atomless (i.e., with no minimal elements) forcing notions are equivalent (see e.g., the hint for \cite[Exercise III.3.70]{Ku:Set}). 


\begin{theo}\label{sumup}
Let $A$ be a simple and unital \cstar-algebra, $m\geq 1$ and let $\bar\varphi$ and $\bar\psi$ be elements of $\pure_m(A)$. Then
\begin{enumerate}
    \item \label{1.sumup} Forcing with $\bE_{A^\circ}(\bar\varphi,\bar\psi)$ adds an approximately inner automorphism $\Phi_\kG$ of $A$ such that $\bar\varphi\circ\Phi_\kG=\bar\psi$.
    \item \label{2.sumup} If $A$ is separable, then $\bE_{A^\circ}(\bar\varphi,\bar\psi)$ is equivalent to the Cohen forcing. 
\end{enumerate}
\end{theo}

\begin{proof}
\eqref{1.sumup} Let $M$ be a countable transitive model of a large enough fragment of $\ax{ZFC}$ such that $\bE_{A^\circ}(\bar\varphi,\bar\psi)$ is an element of $M$, and let $\kG$ be an $M$-generic filter on $\bE_{A^\circ}(\bar\varphi,\bar\psi)$. By Lemma \ref{density}, for any $F\Subset A$, and for all $\varepsilon>0$, there exists some $q\in\bE_{A^\circ}(\bar\varphi,\bar\psi)$ such that if $p\leq q$ then for all $b\in F$, $\|\ad v_p(b)-\ad v_q(b)\|<\varepsilon$ and $\|\ad v_p^*(b)-\ad v_q^*(b)\|<\varepsilon$. Therefore, the nets $\ad v_p$ and $\ad v_p^*$, for $p\in\kG$, are Cauchy with respect to the point-norm topology in $\aut(A)$. By \cite[Lemma 2.6.3]{Fa:STCstar}, $\Phi_L\in\aut(A)$ defined pointwise as $\Phi_L(a):=\lim_\kG\ad v_p(a)$ is an endomorphism of $A$, and its inverse is given by $\Phi_L^{-1}(a)=\lim_\kG\ad v_p^*(a)$ for each $a\in A$. An analogous argument shows that $\Phi_R(a):=\lim_\kG \ad w_p(a)$ for each $a\in A$ is an approximately inner automorphism of $A$.

Let now $a\in A$ be arbitrary and let $\varepsilon>0$. Again using Lemma \ref{density}, choose $p\in\kG$ such that some $b\in F_p\cap G_p$ satisfies $\|a-b\|<\varepsilon/3$ and that $\max\{\varepsilon_p,\delta_p\}<\varepsilon/3$. Then $|\bar\varphi\circ\Phi_L(a)-\bar\psi\circ\Phi_R(a)|<\varepsilon$. Set $\Phi_\kG:=\Phi_L\circ\Phi_R^{-1}$. Since $\varepsilon$ was arbitrary, $\bar\varphi\circ\Phi_\kG=\bar\psi$.

\eqref{2.sumup} Since $A$ is separable, $A^\circ$ is countable. Also, the conditions such that both $\varepsilon_p$ and $\delta_p$ are rational, comprise a dense subset of $\bE_{A^\circ}(\bar\varphi,\bar\psi)$. This set is countable, and by \cite[Proposition 10.20]{Kana:Book}, $\bE_{A^\circ}(\bar\varphi,\bar\psi)$ is equivalent to the Cohen forcing.
\end{proof}

We should point out that the \cstar-algebra $A$ in Theorem~\ref{sumup}  is not required to be separable. If $A$ is (for example) the \cstar-algebra with pure states of characters~$\aleph_0$ and an uncountable $\kappa$  (such $A$ exists by  \cite[Proposition~2.9.4]{Fa:STCstar}) then the existence of an automorphism as guaranteed by Theorem~\ref{sumup} in a generic extension implies that~$\kappa$ is collapsed to~$\aleph_0$. Thus for some nonseparable \cstar-algebras the forcing $\bE_{A^\circ}(\bar\varphi,\bar\psi)$ may collapse cardinals.

The idea of restricting to a countable dense set in order to assure the countable chain condition of a poset used in the proof of Theorem~\ref{sumup} was first used in the context of operator algebras in \cite{Wof:Set}.

\section{The Unique Extension Property of pure states}\label{S.unique}

The most remarkable property of the generic automorphism introduced by the forcing notion $\bEA{\bar\varphi}{\bar\psi}$ is the genericity of its action on $\pure(A)$ (see Theorem \ref{agap}). 

Let $A$ be a simple, unital, and non-type I \cstar-algebra. We will see later on (see Proposition \ref{newpure}) that after forcing with $\bE_{A^\circ}(\bar\varphi,\bar\psi)$, new equivalence classes of pure states of $A$ will appear. The aim of this section is to study the ground-model pure states of $A$ that have a unique pure state extension to $A\rtimes_{\Phi_\kG}\bZ$ in $M[\kG]$. In order to assure that pure states of $A$ in a prescribed set have unique pure extensions to the crossed product mentioned above, we use the tools introduced in \cite{AkeWe:Consistency} and presented in a gory detail in \cite[\S 5.4]{Fa:STCstar}.

By Theorem 3.1 in \cite{kishimoto1981outer}, if $A$ is simple and $\Phi$ is outer, then $A\rtimes_\Phi\bZ$ is simple as well. By Theorem 2 in \cite{AkeWe:Consistency} (see also \cite[Proposition~5.4.7]{Fa:STCstar}), a pure state $\varphi$ on $A$ has a unique extension to a pure state on $A\rtimes_\Phi\bZ$ if, and only if, $\varphi$ is not equivalent to $\varphi\circ\Phi^n$ for all $n\geq 1$. Since the set of all extensions of $\varphi$ is a face in $\st(A\rtimes_\Phi\bZ)$ (see \cite[Lemma 5.4.1]{Fa:STCstar}), $\varphi$ has a unique extension if, and only if, it has a unique pure state extension. If two pure states $\varphi$ and $\psi$ on $A$ have unique extensions $\tilde\varphi$ and $\tilde\psi$ to $A\rtimes_\Phi\bZ$, then these extensions are equivalent if, and only if, $\varphi$ is equivalent to $\psi\circ\Phi^n$ for some $n\in \bbZ$ (this is the case $\Gamma=\bbZ$ of \cite[Theorem~5.4.8]{Fa:STCstar}).

In the following, we use the notation established in Theorem \ref{sumup}. Note that for a fixed $A$ and $A^\circ$, the conditions in all forcings $\bE_{A^\circ}(\bar\varphi,\bar\psi)$ have the same format and that $\bar\varphi$ and $\bar\psi$ behave as side-conditions. We will now relate these forcing notions.

\begin{defin}\label{pwue}
Let $A$ be a unital \cstar-algebra, $m\geq 1$, and let $\bar\varphi$ and $\bar\psi$ be elements of $\pure_m(A)$. We will say that $\bar\varphi$ is \emph{pointwise unitarily equivalent} to $\bar\psi$, in symbols $\bar\varphi\sim_p\bar\psi$, if there exists a tuple $(u_0,\dots,u_{m-1})\in\U(A)^m$ such that for each $i<m$, we have that $\bar\varphi_i\circ\ad u_i=\bar\psi_i$.
\end{defin}


\begin{lem}\label{changeposet} 
Suppose $A$ is a simple and unital \cstar-algebra, $m\geq 1$, and $\bar\varphi$ and $\bar\psi$ are elements of $\pure_m(A)$. Also suppose that $l\geq 0$, $\bar\rho$ and $\bar\sigma$ belong to $\pure_l(A)$, and moreover $\bar\varphi^\frown\bar\rho$ and $\bar\psi^\frown\bar\sigma$ belong to $\pure_{m+l}(A)$. Let us write $\bbP_0:=\bEA{\bar\varphi}{\bar \psi}$, $\bbP_1:=\bEA{\bar\varphi^\frown\bar \rho}{\bar \psi^\frown \bar \sigma}$, and $\leq_j$ for the ordering on $\bbP_j$ for $j<2$. Then 
\begin{enumerate}
\item\label{1.changeposet} Every condition in $\bbP_1$ is a condition in $\bbP_0$. Moreover, if $p$ and $q$ are in $\bbP_1$ then $p\leq_0 q$ if, and only if, $p\leq_1 q$.  

\item\label{2.changeposet} For every $q\in \bbP_0$ there exists some $\bar\rho'\sim_p\bar\rho$ such that some $p\leq_0 q$ belongs to $\bbP_1':=\bEA{\bar\varphi^\frown\bar\rho'}{\bar\psi^\frown\bar\sigma}$.\footnote{In short, $\bbP_1$ is a subordering of $\bbP_0$ and the union of all $\bbP_1'$ as in \eqref{2.changeposet} is dense in $\bbP_0$. This formulation is dangerously misleading, since $\bbP_1$ is typically not a regular subordering of $\bbP_0$.}  
\end{enumerate}
\end{lem}

\begin{proof}
To see that the first part of \eqref{1.changeposet} holds, fix $q\in \bbP_1$. Conditions (1)--(3) of Definition~\ref{poset} do not depend on the tuples of pure states, while (4) and (5) are weakened as one passes to sub-tuples of pure states. The second part of \eqref{1.changeposet} follows because conditions (a)--(c) of Definition~\ref{poset}  do not refer to the pure states.   
 
\eqref{2.changeposet} Fix $q\in\bbP_0$. We will assume that each of  $\bar\rho$ and $\bar\sigma$ consists of a single pure state, $\rho$ and $\sigma$ respectively.  Once this is proved, the general case will follow by induction. By Lemma \ref{goodlemma}, there exists a $((\bar\psi^\frown\sigma)\circ\ad w_q,F_q\cup\ad w_q^*[F_q],\varepsilon_q)$-good pair; denote it $(G_1,\delta_1)$. Since $q$ is a condition, $(G_q,\delta_q)$ is $(\bar\varphi\circ\ad v_q,F_q\cup\ad v_q^*[F_q],\varepsilon_q)$-good, and
\begin{equation*}
    \bar\varphi\circ\ad v_q\approx_{G_q,\delta_q}\bar\psi\circ\ad w_q.
\end{equation*}
Using the goodness of $(G_q,\delta_q)$, choose $v\in\U(A)\cap A^\circ$ to be some unitary such that $\bar\varphi\circ\ad v_qv\approx_{G_1,\delta_1}\bar\psi\circ\ad w_q$ and $\|b-\ad v(b)\|<\varepsilon_q/3$ for all $b\in F_q\cup\ad v_q^*[F_q]$. Set $\varepsilon_p:=\varepsilon_q/3$, $F_p:=F_q$, and $v_p:=v_qv$. Apply now \cite[Proposition 5.2.9]{Fa:STCstar} to find some unitary $u\in\U(A)\cap A^\circ$ such that if $\rho':=\rho\circ\ad u$, then
\begin{equation*}
    (\bar\varphi^\frown\rho')\circ\ad v_p\approx_{G_1,\delta_1}(\bar\psi^\frown\sigma)\circ\ad w_q.  
\end{equation*}
By Lemma \ref{goodlemma}, there is a $((\bar\varphi^\frown\rho')\circ\ad v_p,F_p\cup\ad v_p^*[F_p],\varepsilon_p)$-good pair, denoted $(G_2,\delta_2)$. Also, let $G_p:=G_q\cup G_2$ and $\delta_p:=\min\{\delta_q,\delta_2\}$.

Using now the fact that the pair $(G_1,\delta_1)$ is a good pair, let $w\in\U(A)\cap A^\circ$ be such that $(\bar\varphi^\frown\rho')\circ\ad v_p\approx_{G_p,\delta_p}(\bar\psi^\frown\sigma)\circ\ad w_qw$ and $\|b-\ad w(b)\|<\varepsilon_q/3$ for every $b\in F_q\cup\ad w_q^*[F_q]$. Define $w_p$ as $w_qw$ and set $p$ to be $(F_p,G_p,\varepsilon_p,\delta_p,v_p,w_p)$.

Since the pair $(G_p,\delta_p)$ is $((\bar\varphi^\frown\rho')\circ\ad v_p,F_p\cup\ad v_p^*[F_p],\varepsilon_p)$-good, $p\in\bP_1$. Also, note that if $b\in F_q$ then
\begin{align*}
\|\ad v_q^*(b)-\ad v_p^*(b)\|&=\|\ad v(\ad v_q^*(b))-\ad v_q^*(b)\|\\
&<\varepsilon_q/3<\varepsilon_q-\varepsilon_q/3\leq\varepsilon_q-\varepsilon_p.    
\end{align*}
Also, $\|\ad v_q(b)-\ad v_p(b)\|=\|b-\ad v(b)\|<\varepsilon_q-\varepsilon_p$. The calculations for $w_p$ and $w_q$ are analogous and therefore $p\leq_0 q$.
\end{proof}

\begin{rem}\label{R.reals}
We will follow an abuse of terminology common  in set theory and refer to `reals' as  elements of any uncountable Polish space fixed in advance that is definable in the sense that it belongs to every model of (a sufficiently large fragment of) \ZFC. This is justified by a classical result of Kuratowski, asserting that any two uncountable Polish spaces are Borel-isomorphic. Using properties of the standard coding for Borel sets (see e.g., \cite[p. 504]{Jech:SetTheory}) one sees that the property of being a Borel-isomorphism between two Polish spaces is absolute between transitive models of \ZFCmP. This has as a consequence the fact that a forcing notion adds a new element to $\bbR$ if, and only if, it adds a new element to some (every) uncountable Polish space in the ground model. Thus the phrase `$\bbP$ does not add new reals' is unambiguous, even with the gratuitous use of the phrase `the reals'. 
\end{rem}

A forcing notion $\bP$ adds a new real to the model $M$ if, and only if, it adds a new element to some (equivalently, every) non-trivial \cstar-algebra. Because of this, in a forcing extension $M[\kG]$ we identify $A$ with its completion, and pure states of $A$ with their unique continuous extensions to the completion of $A$. As common in set theory, by $A^M$ we denote the original \cstar-algebra $A$ in $M$ and by $A^{M[\kG]}$ we denote its completion in $M[\kG]$ (see Appendix \ref{S.generic}). Note that $A^M$ is an element of $M[\kG]$ which is an algebra over the field $\bbC^M$. In the extension, the latter is a proper subfield of $\bbC^{M[\kG]}$, hence the former is not a complex algebra, hence not a \cstar-algebra. It is however dense in $A^{M[\kG]}$, which suffices for our purposes. 

\begin{theo}\label{agap}
Suppose $A$ is a simple, unital, separable, and non-type I \cstar-algebra, and let $\bar\varphi$ and $\bar\psi$ be elements of $\pure_m(A)$. If $\rho$ is a pure state on $A$, then some condition in $\bE_{A^\circ}(\bar\varphi,\bar\psi)$ forces that $\rho\circ\Phi^n_\kG$ is equivalent to a ground-model pure state $\sigma$ for some $n\geq 1$ if, and only if, $n=1$ and there exists $i<m$ such that $\rho\sim\varphi_i$ and $\sigma\sim\psi_i$. 
\end{theo}

\begin{proof}[Proof of Theorem~\ref{agap}]
The converse implication is the conclusion of Theorem \ref{sumup}.

The direct implication uses the well-known fact that pure states $\eta$ and $\sigma$ on $A$ that are inequivalent in  a model of a large enough fragment of \ZFC{} remain inequivalent in every forcing extension (and even in every larger transitive model  of a large enough fragment of \ZFC{}). A proof of this fact is included below.

Suppose that the direct implication is false for some $n$, $\rho$, and $\sigma$. Choose the minimal such $n$. Note that if $\rho\sim \varphi_i$ for some $i$  then necessarily $\sigma\sim \psi_i$, and similarly if $\sigma\sim \psi_i$ then $\rho\sim \varphi_i$. 
Since $A$ is separable, by Glimm's theorem (\cite[Corollary~5.5.5]{Fa:STCstar}) it has $2^{\aleph_0}$ inequivalent pure states, if $n>1$ then  we can choose mutually inequivalent pure states $\rho_1,\dots, \rho_{n-1}$ each of which is inequivalent to all $\varphi_i$, $\psi_i$, $\rho$, and $\sigma$.   Therefore, by our assumption, (writing $\rho_0:=\rho$ and $\bar\rho:=(\rho_0,\rho_1,\dots,\rho_{n-1})$) both $\bar\varphi^\frown\bar\rho$ and $\psi^\frown(\rho_1,\dots, \rho_{n-1}, \sigma)$ belong to  $\pure_{m+n}(A)$. 

 Fix $q\in \bE_{A^\circ}(\bar\varphi,\bar\psi)$ which forces that $\rho\circ\Phi^n_\kG\sim\sigma$. By the minimality of $n$, we can extend $q$ so that it forces $\rho\circ \Phi^j_\kG\not\sim\sigma$ for all $1\leq j<n$. Since the ground-model $\U(A)$ is dense in $\U(A)$ of the generic extension, by further extending $q$, we may assume that there exists $u\in\U(A)\cap A^\circ$ such that $q\Vdash\|\rho\circ\Phi^n_\kG-\sigma\circ\ad u\|<1/2$. Using the abundance of inequivalent pure states of $A$ again, let   $\eta\nsim\sigma$ be such that if $\bar\eta:=(\rho_1,\dots,\rho_{n-1},\eta)$, then  $\bar\psi^\frown\bar\eta$ belongs to  $\pure_{m+n}(A)$. By Lemma \ref{changeposet} there exists $\bar\rho'=(\rho'_0, \rho'_1,\dots \rho'_{n-1})$ in $\pure_n(A)$ such that $\bar\rho'\sim_p\bar\rho$ and  some $p\in \bE_{A^\circ}(\bar\varphi^\frown\bar\rho',\bar\psi^\frown\bar\eta)$ extends $q$.

Let $\kH$ be an $M$-generic filter on $\bE_{A^\circ}(\bar\varphi^\frown\bar\rho',\bar\psi^\frown\bar\eta)$ containing $p$. By Theorem~\ref{sumup}, in $M[\kH]$ we have that $\bar\rho'\circ\Phi_\kH=\bar\eta$. Let $v\in\U(A)$ be  such that $\rho=\rho'_0\circ\ad v$ and set $v_\kH:=\Phi_\kH^{-n}(v)$. Since $\ad v\circ\Phi_\kH^n=\Phi_\kH^n\circ\ad v_\kH$, we have 
\begin{equation*}
    \rho\circ\Phi_\kH^n=\rho'_0\circ\ad v\circ\Phi_\kH^n=\rho'_0\circ\Phi_\kH^n\circ\ad v_\kH.
\end{equation*}
Since $\sigma$ is not equivalent to $\eta\circ\ad v_\kH$, by  \cite[Proposition~3.8.1]{Fa:STCstar} some  $a\in A^M_{\leq 1}$ satisfies 
\begin{equation*}
    \left|\left(\eta\circ\ad v_\kH\right)(a)-\left(\sigma\circ\ad u\right)(a)\right|\geq 2.
\end{equation*}
Since $A^M$ is norm-dense in $A^{M[\kH]}$, let $r\in\bE_{A^\circ}(\bar\varphi^\frown\bar\rho',\bar\psi^\frown\bar\eta)$, with $r\leq p$, be such that some $b\in G_r$ satisfies  $\|b-\ad v_\kH(a)\|<1/6$, and $\delta_r<1/6$. The (easy) first part of Lemma \ref{changeposet} implies that $r\in\bE_{A^\circ}(\bar\varphi,\bar\psi)$ and that $r$ extends $q$ as an element of $\bE_{A^\circ}(\bar\varphi,\bar\psi)$. However, $r$ forces (in either of the posets) that
\begin{align*}
    \left|\left(\rho\circ\Phi_\kH^n\right)(a)-\left(\eta\circ\ad v_\kH\right)(a)\right|&=\left|(\rho'_0\circ\Phi_\kH^n)(\ad v_\kH(a))-\eta(\ad v_\kH(a))\right|\\
    &\leq\left|(\rho'_0\circ\Phi_\kH^n)(b)-\eta(b)\right|+2\|b-\ad v_\kH(a)\|<1/2. 
\end{align*}
By \cite[Proposition~3.8.1]{Fa:STCstar}, $\rho\circ\Phi_\kH^n$ is equivalent to $\eta$, and hence $\eta$ and $\sigma$ are equivalent in the extension. But $\eta$ and $\sigma$ are inequivalent ground model pure states. Since $\U(A)^M$ is dense in $\U(A)^{M[\kH]}$, we can choose $v\in\U(A)^M$ such that $\|\eta\circ \Ad v-\sigma\|<2$. By using  
\cite[Proposition~3.8.1]{Fa:STCstar} again we conclude that $\eta$ and $\sigma$ are equivalent in~$M$; contradiction.   
\end{proof}

In the following corollary there is no need to explicitly refer to the ground model. 

\begin{cor}\label{uep}
Suppose that $A$ is a simple, unital, and non-type I \cstar-algebra. Fix elements $\bar\varphi$ and $\bar\psi$ of $\pure_m(A)$, and fix $\rho$ and $\sigma$ in $\pure(A)$. If $\kG$ is a generic filter in $\bEA{\bar\varphi}{\bar\psi}$, then the following statements hold in the forcing extension: 
\begin{enumerate}
\item $\rho$ has multiple pure state extensions to $A\rtimes_{\Phi_\kG}\bZ$ if, and only if, there exists $i<m$ such that $\rho\sim\varphi_i\sim\psi_i$. 
\item If $\rho$ and $\sigma$ are inequivalent and they both have unique pure state extensions to $A\rtimes_{\Phi_\kG} \bZ$ then these extensions are equivalent if, and only if, some $i<m$ satisfies $\rho\sim\varphi_i$ and $\sigma\sim\psi_i$ or $\rho\sim\psi_i$ and $\sigma\sim\varphi_i$.   
\end{enumerate}
\end{cor}

\begin{proof}
By a result from \cite{AkeWe:Consistency} (or see \cite[Proposition~5.4.7]{Fa:STCstar}), a pure state $\zeta$ has a unique pure state extension to the crossed product if, and only if, $\zeta\nsim\zeta\circ \Phi_\kG^n$ for all $n\neq0$. By Theorem \ref{sumup} and Theorem \ref{agap}, this happens if, and only if, $\rho\sim\varphi_i\sim\psi_i$ for some $i<m$. This proves (1). 

In (2) only the direct implication requires a proof. It uses the case when $\Gamma=\bbZ$ of  \cite[Theorem~5.4.8]{Fa:STCstar}. This theorem asserts that  with an action $\alpha$ of a discrete group $\Gamma$ on $A$, two pure states $\rho$ and $\sigma$ on $A$ have equivalent extensions to the reduced crossed product  if and only if  some  $g$ in $\Gamma$ satisfies $\rho \circ \alpha_g \sim \sigma$.  Therefore if $\rho$ and $\sigma$ have unique and equivalent extensions to the reduced crossed product, then $\bEA{\bar\varphi}{\bar\psi}$,  forces that $\rho\sim \sigma\circ \Phi^n$ for some $n\in \bbZ$. Since $\rho$ and $\sigma$ are inequivalent, they remain inequivalent in the forcing extension and therefore $n\neq 0$. By Theorem~\ref{agap}, if $n>0$ then $n=1$, $\rho\sim\varphi_i$, and $\sigma\sim\psi_i$ for some $i$. By the same theorem, if $n<0$ then $n=-1$, $\rho\sim\psi_i$, and $\sigma\sim\varphi_i$. This exhausts the possibilities and concludes the proof.  
\end{proof}

\section{A proof of Theorem \ref{main} from $\diacoh$}

In this section, we introduce our weakening of Jensen's $\diamondsuit_{\aleph_1}$ principle, $\diacoh$, and use it to construct the \cstar-algebra as required in Theorem \ref{main}.  

A subset of $\aleph_1$ is \emph{closed and unbounded} (\emph{club}) if it is unbounded and contains the supremum of each of its bounded subsets.\footnote{In other words, it is unbounded and closed in the ordinal topology.} A subset of $\aleph_1$ is \emph{stationary} if it intersects every club non-trivially. 

Let us first recall Jensen's $\diamondsuit_{\aleph_1}$ principle.

\begin{defin} A \emph{$\diamondsuit_{\aleph_1}$ sequence} is an indexed  family of sets $S_\alpha\subseteq \alpha$, for $\alpha<\aleph_1$, such that for every $X\subseteq \aleph_1$  the set $\{\alpha\mid X\cap \alpha=S_\alpha\}$ is stationary. The $\diamondsuit_{\aleph_1}$ principle asserts that a $\diamondsuit_{\aleph_1}$ sequence exists. 
\end{defin} 

It is well known that $\diamondsuit_{\aleph_1}$ is relatively consistent with $\ax{ZFC}$, that it implies $\ax{CH}$, and that it is not a consequence of $\ax{CH}$ (see \cite{Ku:Set}).

\begin{defin}\label{Def.diacoh} 
A chain $(M_\alpha:\alpha<\aleph_1)$ is a \emph{$\diacoh$-chain} if: 
\begin{enumerate}[$\diacoh$(a)]
    \item Each $M_\alpha$ is a (not necessarily countable) transitive model of \ZFCmP.
    \item\label{2.diacoh} For every $X\subseteq\aleph_1$, the set $\{\alpha<\aleph_1:X\cap\alpha\in M_\alpha\}$ is stationary. 
    \item\label{3.diacoh} For every $\alpha<\aleph_1$, some real in $M_{\alpha+1}$ is Cohen-generic over $M_\alpha$.
\end{enumerate}

We say that the principle \emph{$\diacoh$ holds} if there exists a $\diacoh$-chain.
\end{defin}

Clearly, $\diamondsuit_{\aleph_1}$ implies $\diacoh$. (To see this, fix a $\diamondsuit_{\aleph_1}$ sequence $S_\alpha$, for $\alpha<\aleph_1$, and choose an increasing chain of countable elementary submodels $(N_\alpha)$ of $H(\aleph_2)$ such that $S_\alpha\in N_\alpha$ for all $\alpha$, and assure that $N_{\alpha+1}$ contains a real which is Cohen over~$N_\alpha$; this is possible because $N_\alpha$ is countable.  Take $M_\alpha$ to be the transitive   collapse of $N_\alpha$.)  For a partial converse, note that if each model $M_\alpha$ in a $\diacoh$-chain is countable, or if its intersection with $2^\alpha$ is countable, then $\diamondsuit_{\aleph_1}$ holds. This is a consequence of \cite[Theorem III.7.8]{Ku:Set}. We will discuss the relative consistency of $\diacoh$ in Appendix \ref{S.diamond}.

Let us now discuss the relevance of the condition $\diacoh$\eqref{2.diacoh}. 

First, observe that it implies that $\bigcup_{\alpha<\aleph_1}M_\alpha$ contains all subsets of $\omega$, and therefore all elements of any Polish space (affectionately known as `reals'---see Remark~\ref{R.reals} for both the terminology and the proof). 

Second, this condition applies when $X$ is replaced with any object of cardinality~$\aleph_1$ or with a complete metric space of density character~$\aleph_1$. More specifically, every \cstar-algebra of density character $\kappa$ can be coded by a subset of $\kappa$ (see the introduction to \S3 in \cite{farah2016simple}, or \cite[\S7.1-2]{Fa:STCstar}). Similarly, if $\bar\psi$ is a tuple of states of a \cstar-algebra $B$ of density character $\kappa$, then the structure $(B,\bar\psi)$ can be coded by a subset of $\kappa$. Moreover, the version of L\"owenheim--Skolem theorem for logic of metric structures stated in \cite[Theorem 7.1.4]{Fa:STCstar} implies the following. 

\begin{lem}\label{L.LST}
Suppose that $\kappa$ is a regular and uncountable cardinal, $A=\varinjlim_{\alpha<\kappa} A_\alpha$ is a \cstar-algebra such that the density character of each $A_\alpha$ is strictly smaller than $\kappa$, $A_\beta=\varinjlim_{\alpha<\beta} A_\alpha$ for every limit ordinal $\beta$, $\bar\varphi$ is a tuple of states of $A$, and $X\subseteq\kappa$ is a code for the structure $(A,\bar\varphi)$. Then the set
\begin{equation*}
    \left\{\alpha<\kappa:X\cap\alpha\text{ is a code for }(A_\alpha,\bar\varphi\rs A_\alpha)\right\} 
\end{equation*}
includes a club.\qed
\end{lem}

If $M$ is a transitive model of \ZFCmP{} then we slighty abuse the terminology and say that a \cstar-algebra $B$ \emph{belongs} to $M$ if some code for $B$ belongs to $M$. The analogous remark applies to states of $B$. 

Glimm's dichotomy (see e.g. \cite[Corollary 5.5.8]{Fa:STCstar}) implies that every simple \cstar-algebra $A$ of density character $\kappa<\kc$ either has a unique pure state up to unitary equivalence, in which case $A\cong\sK(\ell_2(\kappa))$, or has $\kc$ unitary equivalence classes of pure states. The conclusion of the following theorem was deduced from $\diamondsuit_{\aleph_1}$ in \cite[Theorem~1.2]{farah2016simple}, as announced in \cite[\S 8.2]{Fa:Logic} (see also \cite[Theorem~11.2.2]{Fa:STCstar}). The special case when $m=1$ (using the full $\diamondsuit_{\aleph_1}$) is the Akemann--Weaver result.

The case when $m=1$ of Theorem \ref{realmain} below is Theorem \ref{main}. In its proof we adopt the approach to $\diamondsuit_{\aleph_1}$ constructions introduced in \cite{rinotdiamond}.

\begin{theo}\label{realmain}
If $\diacoh+\ax{CH}$ holds, then for all $m\geq 1$ there is a simple \cstar-algebra of density character $\aleph_1$ with exactly $m$ pure states up to unitary equivalence that is not isomorphic to any algebra of compact operators on a complex Hilbert space.
\end{theo}

\begin{proof}
Let $(M_\alpha:\alpha<\aleph_1)$ be a $\diacoh$-chain. Using the Continuum Hypothesis, fix  a surjection $f\colon\aleph_1\to H(\aleph_1)$ such that every element of $H(\aleph_1)$ is listed cofinally often. By recursion on $\beta<\aleph_1$, we will define an inductive system of separable, simple, unital, and non-type I \cstar-algebras $A_\beta$. 

Let $A_0$ be a separable, simple, unital, non-type I \cstar-algebra and let $\varphi_i$, for $i<m$, be inequivalent pure states of $A_0$. At the latter stages of the construction we will assure that the following conditions hold for all $\alpha<\aleph_1$.
\begin{enumerate}
    \item If $\xi<\alpha$ then $A_\xi$ is a unital \cstar-subalgebra of $A_\alpha$.
    \item\label{2.main} With $\gamma(\alpha):=\min\{\gamma:A_\alpha\in M_\gamma\}$, $A_{\alpha+1}$ belongs to $M_{\gamma(\alpha)+1}$.\footnote{This function is well-defined: Since $A_\alpha$ is separable, it is coded by a real and therefore belongs to $\bigcup_{\alpha<\aleph_1}M_\alpha$.}
 	\item\label{3.main} Every pure state of $A_\alpha$ that belongs to $M_{\gamma(\alpha)}$ has a unique pure state extension to $A_{\alpha+1}$.
 	\item\label{4.main} If $f(\alpha)$ is a code for a pair $(A_\xi,\psi)$, where $\xi<\alpha$ and $\psi$ is a pure state of $A_\xi$ which has a unique pure state extension to $A_\alpha$, then $\psi$ has a unique pure state extension to $A_{\alpha+1}$, and this extension is equivalent to (the unique pure state extension of) some $\varphi_i$, for $i<m$.  
\end{enumerate}

To describe the recursive construction, suppose that $\beta$ is a countable ordinal such that $A_\alpha$ as required has been defined for all $\alpha<\beta$. 

Consider first the case when $\beta$ is a successor ordinal, $\beta=\alpha+1$. Suppose for a moment that $f(\alpha)$ is a code for a pair $(A_\xi,\psi)$ with the following properties: 
\begin{enumerate}
    \item[(a)] $\xi<\alpha$. 
    \item[(b)] $\psi$ is a pure state of $A_\xi$ that has a unique extension $\tilde\psi$ to a pure state of $A_\alpha$. 
    \item[(c)] For all $i<m$, $\tilde\psi$ is inequivalent to the unique extension of $\varphi_i$ to $A_\alpha$ (still denoted $\varphi_i$). 
\end{enumerate}

By the second part of Theorem \ref{sumup}, $\bE_{A^\circ_\alpha}(\varphi_0,\tilde\psi)$ is forcing-equivalent to the poset for adding a single Cohen real. Since $M_{\gamma(\alpha)+1}$ contains a real that is Cohen-generic over $M_{\gamma(\alpha)}$, by \cite[Lemma IV.4.7]{Ku:Set}, it contains an $M_{\gamma(\alpha)}$-generic filter $\kG$ on $\bE_{A^\circ_\alpha}(\varphi_0,\tilde\psi)$. By the first part of Theorem \ref{sumup}, $\Phi_{\kG}$ is an approximately inner automorphism of $A_\alpha$ such that $\varphi_0\circ\Phi_{\kG}=\tilde\psi$. By Corollary \ref{uep}, the \cstar-algebra $A_{\alpha+1}:=A_\alpha\rtimes_{\Phi_{\kG}}\bZ$ has the property that every pure state of $A_\alpha$ that belongs to $M_{\gamma(\alpha)}$ has a unique pure state extension to $A_{\alpha+1}$. By the second part of Corollary~\ref{uep}, the unique pure state extensions of $\varphi_i$, for $i<m$, to $A_{\alpha+1}$ are inequivalent. Also, $A_{\alpha+1}$ is separable, simple, unital, and non-type I by Corollary \ref{bacan}.  

If $f(\alpha)$ does not satisfy the conditions (a)--(c), let $A_{\alpha+1}:=A_\alpha$. 

If $\beta$ is a limit ordinal, take $A_\beta:=\varinjlim_{\alpha<\beta} A_\alpha$.

Finally, let $A_{\aleph_1}:=\varinjlim_{\alpha<\aleph_1} A_\alpha$.

By the construction described above, each one of the the pure states $\varphi_i$, for $i<m$, of $A_0$ has a unique pure state extension to $A_{\aleph_1}$, and these pure state extensions are inequivalent.

Suppose that $\psi$ is a pure state of $A_{\aleph_1}$. In order to prove that it is equivalent to $\varphi_i$ for some $i<m$, fix a code $X\subseteq\aleph_1$ for the pair $(A_{\aleph_1},\psi)$. By \cite[Proposition~7.3.10]{Fa:STCstar}, the set $\{\alpha<\aleph_1:\psi\rs A_\alpha\text{ is pure}\}$ is a club, and by Lemma \ref{L.LST}, the set
\begin{equation*}
    \{\alpha<\aleph_1:X\cap \alpha\text{ is a code for }(A_\alpha,\psi\rs A_\alpha)\}    
\end{equation*}
is a club as well. By $\diacoh$\eqref{2.diacoh}, there exists some $\alpha$ in the intersection of these two clubs such that $X\cap\alpha\in M_\alpha$. In particular, both $A_\alpha$ and $\psi\rs A_\alpha$ belong to $M_\alpha$, i.e., $\gamma(\alpha)=\alpha$.

\begin{claim}
The pure state $\psi\rs A_\alpha$ of $A_\alpha$ has a unique pure state extension to $A_\beta$ for every $\beta>\alpha$.
\end{claim}

\begin{proof}
This is proved by induction on $\beta$: At the successor stages, it follows from properties \eqref{2.main} and \eqref{3.main}. At the limits, note that the unique pure state extension of $\psi\rs A_\alpha$ is definable from $\psi\rs A_\xi$, for $\alpha<\xi<\beta$, and therefore belongs to the relevant model. 
\end{proof}

By the choice of the function $f$, there exists some $\beta<\aleph_1$ such that $f(\beta)$ codes the pair $(A_\alpha,\psi\rs {A_\alpha}$). By the definition of $A_{\beta+1}$, the restrictions of $\psi$ and some $\varphi_i$, for $i<m$, to $A_{\beta+1}$ are equivalent. This proves that $A_{\aleph_1}$ has exactly $m$ inequivalent pure states. Since $A_{\aleph_1}$ is infinite-dimensional and unital, it is not isomorphic to any algebra of compact operators.
\end{proof}

The proof of Theorem \ref{main} will be completed in Appendix \ref{S.diamond}. In this section, we will prove that $\diacoh+\ax{CH}$ is relatively consistent with the negation of $\diamondsuit_{\aleph_1}$. Once proven, this will provide a model of $\ax{ZFC}$ in which both $\diamondsuit_{\aleph_1}$ and Glimm's dichotomy fail.

\section{A proof of Theorem~\ref{C.separable}, part I: Extending a GNS representation gently}\label{S.X}

This section contains finer analysis of the forcing notion $\bEA{\bar\varphi}{\bar\psi}$, culminating in Lemma \ref{L.Extension}. The following is an analog of Theorem \ref{agap}. 

\begin{theo}\label{preglimm}
Suppose that $\Theta$ is an outer automorphism of a separable, simple, unital, non-type I \cstar-algebra $A$, $m\geq 1$, $\bar\varphi$ and $\bar\psi$ belong to $\pure_m(A)$, and $\rho$ is a pure state of $A$ inequivalent to all $\varphi_i$, for $i<m$. If $\Theta_\kG$ is defined as $\Phi_\kG\circ\Theta\circ\Phi_\kG^{-1}$, then $\bE_{A^\circ}(\bar\varphi,\bar\psi)$ forces that $\rho\circ\Theta_\kG$ is inequivalent to any ground-model pure state of $A$. 
\end{theo}

\begin{proof}
Towards obtaining a contradiction, assume that in $M[\kG]$ we have $\rho\circ\Theta_\kG\sim\sigma$ for a ground-model pure state $\sigma$. Then fix $u\in\U(A)\cap A^\circ$ and $q\in\bE_{A^\circ}(\bar\varphi,\bar\psi)$ such that $q$ forces that $\|\rho\circ\Theta_\kG -\sigma \circ \Ad u\|<1/2$. This implies that
\[
\|\rho\circ \Phi_\kG\circ\Theta-\sigma \circ \Ad u \circ \Phi_\kG\|<1/2.
\]
We first consider the most difficult case, when $\rho$ is equivalent to $\sigma$. 

Since $\Theta$ is outer, by the slight extension of \cite[Theorem 2.1]{kishimoto1981outer} proved in \cite[Theorem~2.4]{farah2016simple}, there exists an uncountable set of pure states $\eta$ of $A$ each of which satisfies $\eta\circ\Theta\nsim\eta$. We can therefore choose a pure state $\eta$ such that $\eta\nsim\eta\circ\Theta$ and each one of $\eta$ and $\eta\circ \Theta$ is inequivalent to all $\psi_i$, for $i<m$. By Lemma \ref{changeposet}, there are a pure state $\rho'\sim\rho$ and a condition $p\leq q$ in $\bEA{\bar\varphi}{\bar\psi}$ such that $p$ also belongs to the forcing notion $\bEA{\bar\varphi^\frown\rho'}{\bar\psi^\frown\eta}$. 

The remainder is analogous to the corresponding part of the proof of Theorem~\ref{agap}: if $\kH$ is an $M$-generic filter on $\bE_{A^\circ}(\bar\varphi^\frown\rho',\bar\psi^\frown\eta)$ containing $p$ then, in $M[\kH]$, $\rho'\circ\Phi_\kH=\eta$. Let $v\in\U(A)$ be such that $\rho=\rho'\circ\ad v$ and set $v_\kH:=\Phi_\kH^{-1}(v)$, so that $\rho\circ\Phi_\kH=\rho'\circ\ad v\circ\Phi_\kH=\rho'\circ\Phi_\kH\circ\ad v_\kH$. Since $\eta\nsim\eta\circ\Theta$, we have that $\eta\circ\ad v_\kH\circ \Theta\nsim\eta\circ\ad v_\kH$. Because of this, we can find $a\in A_{\leq 1}^M$ such that
\begin{equation*}
    \left|\left(\eta\circ\ad v_\kH\circ\ad u\right)(a)-\left(\eta\circ\ad v_\kH\circ\Theta\right)(a)\right|\geq 3/2. 
\end{equation*}

Since $A^M$ is norm-dense in $A^{M[\kH]}$, there is a condition $r\in\bE_{A^\circ}(\bar\varphi^\frown\rho',\bar\psi^\frown\eta)$ extending $p$, with $\delta_r<1/6$, and such that some $b$ and $c$ in $G_r$ satisfy
\begin{equation*}
    \max\left\{\|b-\ad v_\kH u(a)\|,\|c-\ad v_\kH(\Theta(a))\|\right\}<1/6.
\end{equation*}
By the easy part of Lemma \ref{changeposet}, $r\in\bE_{A^\circ}(\bar\varphi,\bar\psi)$ and it extends $q$ in this poset. From the choice of $r$, we can conclude that, in $\bE_{A^\circ}(\bar\varphi,\bar\psi)$, $r$ forces both 
\begin{align*}
    |(\rho\circ\Phi_\kH\circ\ad u)(a)-(\eta\circ&\ad v_\kH\circ\ad u)(a)|\\
    &=|(\rho'\circ\Phi_\kH)(\ad v_\kH u(a))-\eta(\ad v_\kH u(a))|\\
    &\leq|(\rho'\circ\Phi_\kH)(b)-\eta(b)|+2\|b-\ad v_\kH u(a)\|<1/2
\end{align*}
and 
\begin{align*}
    |(\rho\circ\Phi_\kH\circ\Theta)(a)-(\eta&\circ\ad v_\kH\circ \Theta)(a)|\\
    &=|(\rho'\circ\Phi_\kH)(\ad v_\kH(\Theta(a)))-\eta(\ad v_\kH(\Theta(a)))|\\
    &\leq|(\rho'\circ\Phi_\kH)(c)-\eta(c)|+2\|c-\ad v_\kH(\Theta(a))\|<1/2.  
\end{align*}
By the triangle inequality and our choice of $a$, we obtain $1/2+1/2> 3/2$; contradiction. This concludes the discussion of the case when $\rho\sim \sigma$.

Suppose now that $\rho\nsim\sigma$. As in the first case, in each of the two subcases of this case (corresponding to  (1) and (2) below), we will use Lemma \ref{changeposet} to define a forcing notion $\bbP$ and a condition $p\leq q$ in $\bE_{A^\circ}(\bar\varphi,\bar\psi)$ that also belongs to $\bbP$:

\begin{enumerate}
    \item If $\sigma\sim \varphi_i$ for some $i<m$, choose a pure state $\zeta$ that is not equivalent to any of the $\psi_j$ and such that in addition $\zeta\circ\Theta\nsim \psi_i$. This is possible because $A$ has $2^{\aleph_0}$ inequivalent pure states. By Lemma \ref{changeposet}, there are is a $\rho'\sim \rho$ and a condition $p\leq q$ in the poset $\bbP:=\bEA{\bar\varphi^\frown \rho'}{\bar\psi^\frown\zeta}$.
    
    \item If $\sigma$ is not equivalent to any of the $\varphi_i$, choose two pure states, $\zeta$ and $\eta$, that are not equivalent to any of the $\psi_j$ and such that in addition $\zeta\circ\Theta\nsim\eta$. By Lemma \ref{changeposet}, there are pure states $\rho'\sim \rho$, $\sigma'\sim \sigma$, and a condition $p\leq q$ in the poset $\bbP:=\bEA{\bar\varphi^\frown\rho'^{\frown}\sigma'} {\bar\psi^\frown\zeta^\frown\eta}$.
\end{enumerate}
In each of these two cases the proof that the assumptions lead to a contradiction is analogous to the proof in the case when $\rho\sim\sigma$ and is therefore omitted.  
\end{proof}

\begin{cor}\label{uepsiete}
Suppose that $A$ is a separable, simple, unital \cstar-algebra, $\Theta$ is an outer automorphism of $A$ of order two, $m\geq 1$, $\bar\varphi$ and $\bar\psi$ belong to $\pure_m(A)$, and $\rho$ is a pure state of $A$ inequivalent to all the $\varphi_i$, for $i<m$. If $\Theta_\kG:=\Phi_\kG\circ\Theta\circ\Phi_\kG^{-1}$, then $\rho$ has multiple pure state extensions to $A\rtimes_{\Theta_\kG}\bZ/2\bZ$ if, and only if, there exists some $i<m$ such that $\rho\sim\varphi_i\sim\psi_i$.
\end{cor}

\begin{proof} As in the proof of Corollary~\ref{uep}, by Theorem~\ref{sumup} only the direct implication requires a proof. We prove its contrapositive, whose proof is analogous to the proof of Corollary \ref{uep}.   Assume that there is no $i$ such that $\rho\sim \varphi_i \sim \psi_i$. Then Theorem~\ref{preglimm} (with $\sigma=\rho$) implies that  $\rho\circ\Theta_\kG$ is inequivalent to $\rho$. Since $\Theta_\kG$ has order two, \cite[Theorem~5.4.8]{Fa:STCstar} implies that  $\rho$ has a unique pure state extension to the reduced crossed product. 
\end{proof}

In order to prove Theorem \ref{C.separable}, we need to take a closer look at the inner workings of the GNS construction (see \cite[\S 1.10]{Fa:STCstar}). If $\varphi$ is a state on a \cstar-algebra $A$, then it defines a sesquilinear form on $A$ by $(a|b)_\varphi:=\varphi(b^*a)$. The completion of $A$ with respect to this norm is a Hilbert space $\ell_2(A,\varphi)$ (denoted by $H_\varphi$ in \cite{Fa:STCstar} and~\S\ref{subsecstates}), and the representation $\pi_\varphi$ is defined by the left multiplication. If $A$ is a \cstar-subalgebra of $B$ and $\tilde\varphi$ is a state on $B$ that extends $\varphi$, then $\ell_2(A,\varphi)$ is naturally identified with a closed subspace of $\ell_2(B,\tilde\varphi)$. 

\begin{lem}\label{L.PureExtensions}
Let $A$ be a unital \cstar-algebra with an outer automorphism $\Phi$ such that $\Phi^n=\id_A$ for some $n\geq 2$. Let $B:=A\rtimes_\Phi \bbZ/n\bbZ$, and let $u\in B$ be the unitary such that $\Phi(a)=\ad u(a)$ for every $a\in A$. Suppose that  $\varphi$ is a state (not necessarily pure) on $A$ such that $\varphi=\varphi\circ\Phi$ and that $\psi\in \st(B)$ is an extension of $\varphi$ satisfying $\psi(u)^n=1$. Then the following conclusions hold. 
\begin{enumerate}
    \item\label{1.L.PureExtension} For every $a_0,\dots,a_{n-1}\in A$, we have that
    \begin{equation*}
        \psi\left(\sum_{j<n}a_ju^j\right)=\sum_{j<n}\psi(u)^j\varphi(a_j),
    \end{equation*}
    hence $\psi$ is uniquely determined by $\varphi$ and $\psi(u)$.
	
	\item\label{2.L.PureExtension} $\ell_2(A,\varphi)=\ell_2(B,\psi)$.
	
	\item\label{3.L.PureExtension} If $\varphi$ is pure, so is $\psi$.
\end{enumerate}
\end{lem}

\begin{proof}
We first prove that if $\psi(u)$ belongs to the unit circle then   $\varphi\circ \Phi=\varphi$. For every $a\in A$, we have $\Phi(a)=uau^*$. Since $\psi$ is a state, it is  self-adjoint, hence $\psi(u^*)=\overline{\psi(u)}$. Finally, since the spectrum of $u$ is included in $\bbT$,  $\psi(u)$ is an extreme point of this spectrum, and therefore \cite[Propositon~1.7.8]{Fa:STCstar} applied twice implies   
\[
\varphi(\Phi(a))=\psi(uau^*)=\psi(u)\psi(au^*)=\psi(u)\psi(a)\psi(u^*)=\psi(a)\psi(u)\overline{\psi(u)}=\varphi(a). 
\]

(\ref{1.L.PureExtension}) Since $\psi(u)$ is an element of the unit circle and $u$ is a unitary, by \cite[Proposition 1.7.8]{Fa:STCstar}, for every continuous, complex-valued function $f$ defined on the unit circle and every $a\in A$ we have that $\psi(af(u))=f(\psi(u))\psi(a)=f(\psi(u))\varphi(a)$. 
Therefore $\psi(u^j)=\psi(u)^j$ and $\psi(au^j)=\psi(u)^j\varphi(a)$ for all $j$ and all $a\in A$. 
Thus    (\ref{1.L.PureExtension}) follows by the additivity of $\psi$.   For use in the proof of \ref{2.L.PureExtension}, we note that, since $\Phi^n=\id_A$,  every element of the crossed product $B$ is of the form $\sum_{j<n}a_ju^j$, 

(\ref{2.L.PureExtension}) By the GNS construction (see \cite[\S 1.10]{Fa:STCstar}), to prove that $\ell_2(A,\varphi)=\ell_2(B,\psi)$, it is enough to prove that for every $a\in A$ and every $k<n$, there exists some $a'\in A$ such that $\|a'-au^k\|_\psi=0$. Once proven, this will imply that a dense subspace of $\ell_2(B,\psi)$ is included in $\ell_2(A,\varphi)$,  and therefore the spaces coincide.  We claim that $a':=\psi(u)^k a$ is as desired:
\begin{align*}
    \left\|\psi(u)^k a-au^k\right\|^2_\psi&=\psi\left(\left(\psi(u)^k a-au^k\right)^*\left(\psi(u)^k a-au^k\right)\right)\\
    &=\psi\left(\left(\psi(u)^{-k}-u^{-k}\right)a^*a\left(\psi(u)^k-u^k\right)\right). 
\end{align*}
Now, $\left(\psi(u)^{-k}-u^{-k}\right)a^*a\left(\psi(u)^k-u^k\right)\leq 
\|a^*a\|\left(\psi(u)^{-k}-u^{-k}\right)\left(\psi(u)^k-u^k\right)$, and  a simple calculation using $\psi(u^k)=\psi(u)^k$ proven in (\ref{1.L.PureExtension}) gives
\[
\psi(\left(\psi(u)^{-k}-u^{-k}\right)\left(\psi(u)^k-u^k\right))=0.
\]
Since $\psi$ is positive, by the earlier calculation this implies 
$\left\|\psi(u)^k a-au^k\right\|^2_\psi=0$ as required.

(\ref{3.L.PureExtension}) If $\varphi$ is pure, then an extension $\psi$ of $\varphi$ is pure if, and only if, it is an extreme point of the set $\mathcal{E}:=\{\theta\in\st(B):\theta\text{ extends }\varphi\}$ by \cite[Lemma 5.4.1]{Fa:STCstar}. Let $g\colon\mathcal{E}\to\bC$ be given by $g(\theta):=\theta(u)$. Then  $g[\mathcal{E}]$ is included in the convex closure of the spectrum of $u$, $\textrm{co}(\textrm{sp}(u))$. Since $g$ is an affine and  weak$^*$-continuous function on a closed face of $\st(B)$,  $g[\mathcal{E}]$ is a compact, convex subset of $\textrm{co}(\textrm{sp}(u))$, and the $g$-preimage of a face of $g[\mathcal{E}]$ is a face of $\mathcal{E}$. But, being on the unit circle,  $\psi(u)$ is an extreme point of $g[\mathcal{E}]$, and   (\ref{1.L.PureExtension}) implies that $\psi$ is uniquely determined by $\psi(u)$. Therefore $g^{-1}(\{\psi(u)\})$ is a singleton. Its unique element $\psi$ is an extreme point of $\mathcal{E}$,  and therefore  a pure state by \cite[Lemma 5.4.1]{Fa:STCstar}.
\end{proof}

Lemma \ref{L.Extension} below is based on \cite[Lemma~2.7]{farah2016simple}. The key property of the CAR algebra $\CAR$ used in it is extracted in the following lemma implicit in \cite{farah2016simple}.  

\begin{lem} \label{L.ordertwo}
There are inequivalent pure states $\rho_j$, $\sigma_j$, $\eta_j$ for $j\in \bbN$ on~$\CAR$ and an automorphism $\Theta$ of $\CAR$ of order two such that the following conditions hold. 
\begin{enumerate}
\item $\sigma_j=\rho_j\circ \Theta$ and $\eta_j=\eta_j\circ \Theta$ for all $j$.  
\item $\CAR\rtimes_\Theta \bbZ/2\bbZ$ is isomorphic to $\CAR$. 
\end{enumerate}
\end{lem}

\begin{proof}
Identify $\CAR$ with $\bigotimes_\bbN A_n$, where $A_n\cong M_{2^n}(\bbC)$. Let $\varphi_j$, for $j\in \bbN$, be a family of separated product states of $\CAR$ (see \cite[Definition 2.5]{farah2016simple}). The existence of such family is guaranteed by \cite[Lemma 2.6, (1) implies (2)]{farah2016simple}. Let $u_n$ be a self-adjoint unitary in $A_n$ as defined in the proof of \cite[Lemma 2.7]{farah2016simple}, so that for every $n$ the projections in $A_n$ separating the pure states satisfy the analogues of conditions (6)--(8). Let $\Theta:=\bigotimes_\bbN \Ad u_n$. Then the action of $\Theta$ on the distinguished pure states is as required. To complete the proof, note that as in \cite[Lemma 2.7]{farah2016simple}, the classification of AF algebras implies that $\CAR\rtimes_\Theta \bbZ/2\bbZ$ is isomorphic to $\CAR$. 
\end{proof}

\begin{lem}\label{L.Extension}
Suppose that $X$, $Y$, and $Z$ are disjoint finite sets of pure states of $\CAR$ and $\sfF\colon X\to Y$ is a bijection. Then there are $\bar\varphi$ and $\bar\psi$ in $\pure_{m+l}(\CAR)$ (where $m=|X|$ and $l=|Z|$), and $\Theta\in \Aut (\CAR)$ such that $\bE_{\CAR^\circ}(\bar\varphi,\bar\psi)$ forces the following: 
\begin{enumerate}
    \item\label{1.L.Extension} With $\Theta_\kG:=\Phi_\kG\circ\Theta\circ\Phi_\kG^{-1}$, we have that $B:=\CAR\rtimes_{\Theta_\kG}\bZ/2\bZ$ is isomorphic to $\CAR$. 
    
    \item \label{2.L.Extension} Every $\eta\in Z$ has exactly two pure state extensions, denoted $\eta_{+1}$ and $\eta_{-1}$, to $B$, and $\ell_2(A,\eta)=\ell_2(B,\eta_{\pm 1})$.  
    
    \item\label{3.L.Extension} For every $\eta\in X$, $\eta$ and $F(\eta)$ have unique pure state extensions to $B$, and these extensions are equivalent. 
\end{enumerate}
\end{lem}
 
\begin{proof}
For convenience, we write $A:=\CAR$. The plan is to match the pure states in $X$, $Y$, and $Z$ to those provided by Lemma \ref{L.ordertwo} and import $\Theta$ from there. More specifically, fix $\rho_j$, $\sigma_j$, $\eta_j$, and $\Theta$ as guaranteed by Lemma \ref{L.ordertwo}. Enumerate $X$ as $\varphi_j$, for $j<m$, and let $\varphi_{m+j}:=\sfF(\varphi_j)$ for $j<m$. Enumerate $Y$ as $\varphi_{2m+j}$, for $j<l$. Now let $\psi_j:=\rho_j$ and $\psi_{m+j}:=\sigma_j$ if $j<m$, and let $\psi_{2m+j}:=\eta_j$ if $j<k$. 
\begin{enumerate}
    \item Since $\Theta_\kG$ is conjugate to $\Theta$, we have $B\cong\CAR$.
    
    \item Fix $\eta\in Z$. Then $\eta=\varphi_i=\psi_i$ for some $i$, and therefore Lemma \ref{L.PureExtensions} implies that $\eta$ has exactly two pure state extensions, $\eta_{\pm 1}$, to $B$ and that $\ell_2(A,\eta)=\ell_2(B,\eta_{\pm 1})$.
    
    \item If $\eta\in X$ and $\zeta:=\sfF(\eta)$, then $\varphi_i=\eta$ and $\psi_i=\zeta$ for some $i$. Theorem \ref{agap} implies that $\eta$ and $\zeta$ have unique pure state extensions to $B$, and Theorem \ref{sumup} implies that they are equivalent.
\end{enumerate}
This concludes the proof.  
\end{proof}

\section{A proof of Theorem~\ref{C.separable}, part II: The diagonalization argument}

In this section, we prove that $\diacoh+\ax{CH}$ implies that the conclusion of Glimm's dichotomy fails even for separably represented \cstar-algebras. More precisely, we prove the following: 

\begin{theo} \label{T.separable}
Assume $\diacoh+\ax{CH}$. For every $m\geq 0$ and every $n\geq 1$ there exists a simple, unital \cstar-algebra of density character $\aleph_1$ with exactly $m+n$ unitary equivalence classes of irreducible representations such that $m$ of these representations are on a separable Hilbert space and $n$ of these representations are on a non-separable Hilbert space. 
\end{theo}

A simple \cstar-algebra with irreducible representations on both separable and non-separable Hilbert spaces can be constructed in $\ax{ZFC}$ (see \cite[Theorem~10.4.3]{Fa:STCstar}). Both this \cstar-algebra and the one in Theorem \ref{T.separable} are inductive limits of inductive systems of \cstar-algebras all of which are isomorphic to the CAR algebra.

\begin{proof}
With Lemma \ref{L.Extension} at our disposal, this proof is analogous to that of Theorem \ref{realmain}. Fix a $\diacoh$-chain $(M_\alpha:\alpha<\aleph_1)$. Using the Continuum Hypothesis, fix a surjection $f\colon\aleph_1\to H(\aleph_1)$ such that every element of $H(\aleph_1)$ is listed cofinally often. By recursion on $\alpha<\aleph_1$, we will define an inductive system of separable, simple, unital, non-type I $\C$-algebras, $A_\alpha$. For each $A_\alpha$ we will have a distinguished $(m+n)$-tuple of inequivalent pure states, $(\varphi^\alpha_i:i<m+n)$. 

Let $A_0$ be the CAR algebra with inequivalent pure states $\varphi^0_i$, for $i<m+n$. At the latter steps of the construction, we will assure that for all $\alpha<\aleph_1$ the following conditions hold: 
\begin{enumerate}
    \item If $\xi<\alpha$ then $A_\xi$ is a unital \cstar-subalgebra of $A_\alpha$.
    \item\label{x.2.main} With $\gamma(\alpha):=\min\{\gamma:A_\alpha\in M_\gamma\}$, $A_{\alpha+1}$ belongs to $M_{\gamma(\alpha)+1}$.
 	\item\label{x.3.main} Every pure state $\psi$ of $A_\alpha$ that belongs to $M_{\gamma(\alpha)}$, except $\varphi^\alpha_i$, for $i<m$, has a unique pure state extension to $A_{\alpha+1}$.
 	\item\label{x.4.main} If $f(\alpha)$ is a code for a pair $(A_\xi,\psi)$, where $\xi<\alpha$, $\psi$ is a pure state of $A_\xi$ which has a unique pure state extension $\tilde\psi$ to $A_\alpha$, and $\tilde\psi$ is inequivalent to $\varphi^\alpha_i$ for all $i<m+n$, then $\psi$ has a unique pure state extension to $A_{\alpha+1}$, and this extension is equivalent to $\varphi^{\alpha+1}_m$.  
	\item\label{x.5.main} For all $i<m$, $\varphi_i^{\alpha+1}$ extends $\varphi_i^\alpha$ and $\ell_2(A_{\alpha+1},\varphi^{\alpha+1}_i)=\ell_2(A_\alpha,\varphi^\alpha_i)$.\footnote{For the notation, see the discussion preceding Lemma \ref{L.PureExtensions}.}  
\end{enumerate}

In order to describe the recursive construction, suppose that $\beta$ is a countable ordinal such that $A_\alpha$ as required has been defined for all $\alpha<\beta$. As in the proof of Theorem \ref{realmain}, the interesting case is when $\beta=\alpha+1$ for some $\alpha$ and $f(\alpha)$ is a code for a pair $(A_\xi,\psi)$ that satisfies the following conditions:
\begin{enumerate}
    \item[(a)] $\xi<\alpha$. 
    \item[(b)] $\psi$ is a pure state of $A_\xi$ that has a unique extension $\tilde\psi$ to a pure state of $A_\alpha$. 
    \item[(c)] For all $i<m+n$, $\tilde\psi$ is inequivalent to $\varphi^\alpha_i$.
\end{enumerate}

By the second part of Theorem \ref{sumup}, any forcing notion of the form $\bE_{A^\circ_\alpha}(\bar\rho,\bar\sigma)$ is forcing-equivalent to the forcing notion for adding a single Cohen real. Since $M_{\gamma(\alpha)+1}$ contains a real that is Cohen-generic over $M_{\gamma(\alpha)}$, by \cite[Lemma IV.4.7]{Ku:Set}, it contains an $M_{\gamma(\alpha)}$-generic filter for any forcing notion of this form. Therefore, Lemma \ref{L.Extension} implies that in $M_{\gamma(\alpha)+1}$ there exists an automorphism $\Theta_\kG$ of $A_\alpha$ of order two and such that the \cstar-algebra $A_{\alpha+1}:=A_\alpha\rtimes_{\Theta_\kG}\bZ/2\bZ$ is isomorphic to the CAR algebra, each $\varphi^\alpha_i$ for $m\leq i<m+n$ has a unique pure state extension to $A_{\alpha+1}$, $\psi$ has a unique pure state extension to $A_{\alpha+1}$ equivalent to\footnote{Note that $n\geq 1$, hence $\varphi^\alpha_m$ is well-defined for every $\alpha<\aleph_1$.} $\varphi^{\alpha+1}_m$, and $\varphi^\alpha_i\circ \Theta_\kG =\varphi^\alpha_i$ for $i<m$. Lemma \ref{L.PureExtensions} implies that, for $i<m$, $\varphi^\alpha_i$ has a pure state extension $\varphi^{\alpha+1}_i$ to $A_{\alpha+1}$ that satisfies $\ell_2(A_\alpha,\varphi^\alpha_i)=\ell_2(A_{\alpha+1},\varphi_i^{\alpha+1})$. By Corollary \ref{uepsiete}, any other pure states of $A_\alpha$ that belongs to $M_{\gamma(\alpha)}$ has a unique pure state extension to $A_{\alpha+1}$. Also, $A_{\alpha+1}$ is separable, simple, unital, and non-type I by Corollary \ref{bacan}. 

If $f(\alpha)$ does not satisfy the conditions (a)--(c), let $A_{\alpha+1}:=A_\alpha$ and, for each $i<m+n$, let $\varphi^{\alpha+1}_i:=\varphi^\alpha_i$. 

If $\beta$ is a limit ordinal, take $A_\beta:=\varinjlim_{\alpha<\beta} A_\alpha$ and, for each $i<m+n$, define $\varphi_i^\beta$ as the unique pure state of $A_\beta$ that extends $\varphi^\alpha_i$ for all $\alpha<\beta$. Since $\varphi^\beta_i$ is definable from its restrictions, it belongs to the relevant model.

This describes the recursive construction.

Let $A_{\aleph_1}:=\varinjlim_{\alpha<\aleph_1}A_\alpha$, and for $i<m+n$ let $\varphi_i$ be the unique pure state of $A_{\aleph_1}$ that extends $\varphi^\alpha_i$ for all $\alpha<\aleph_1$. 

By the construction, the pure states $\varphi_i$, for $i<m+n$, are inequivalent. By \eqref{x.5.main} and induction, for $i<m$, we have $\ell_2(A_{\aleph_1},\varphi_i)=\ell_2(A_0,\varphi_i^0)$, and therefore the GNS Hilbert space associated with $\varphi_i$ is separable. If $m\leq i<m+n$, then $\varphi^{\alpha+1}_i$ is the unique extension of $\varphi^\alpha_i$ and therefore Lemma \ref{L.Extension} implies that $\ell_2(A_\alpha,\varphi^\alpha_i)$ is a proper subspace of $\ell_2(A_{\alpha+1},\varphi^{\alpha+1}_i)$ for all $\alpha<\aleph_1$. Therefore, the GNS Hilbert space associated with $\varphi_i$ is non-separable.

It only remains to prove that every pure state of $A_{\aleph_1}$ is equivalent to some $\varphi_i$, but this proof is analogous to the corresponding proof in Theorem \ref{realmain} and therefore omitted.
\end{proof}

The reader may wonder whether it is possible to sharpen the conclusion of Theorem \ref{T.separable} and obtain a simple, unital, infinite-dimensional \cstar-algebra $A$ with at most $m\leq \aleph_0$ irreducible representations up to unitary equivalence such that every irreducible representation of $A$ is on a separable Hilbert space. The answer is well-known to be negative in the case when $m=1$ (it is Rosenberg's result, see e.g., \cite[Corollary~5.5.6]{Fa:STCstar}, that a counterexample to Naimark's problem cannot be separably represented). A proof analogous to that of Rosenberg's result provides a negative answer in the general case. 
 
\begin{prop}
Suppose that $A$ is a non-type I \cstar-algebra all of whose irreducible representations are on a separable Hilbert space. Then $A$ has at least $2^{\aleph_0}$ spatially inequivalent irreducible representations. 
\end{prop}
 
\begin{proof}
The assumption on $A$ is used only to prove that it has a self-adjoint element~$a$ whose spectrum is a perfect set. In $\CAR$, there exists a positive contraction $a_0$ with this property. To see this, note that the diagonal masa (i.e., the maximal abelian subalgebra that is the inductive limit of the algebras of diagonal matrices) is isomorphic to the algebra of continuous functions on the Cantor space, and let $a_0$ correspond to the identity map on the Cantor space via the continuous functional calculus. By Glimm's theorem  (\cite[Theorem~1.7.2]{Fa:STCstar}, but see also \cite{Fa:Errata}), $A$ has a subalgebra whose quotient is isomorphic to $\CAR$. Let $a$ be a self-adjoint lift of $a_0$ to $A$ (it exists by~\cite[\S 2.5]{Fa:STCstar}). Then the spectrum of $a$ includes the spectrum of $a_0$, hence $a$ is as required. 
  
For every element $x$ of the spectrum of $a$, fix a pure state $\varphi_x$ on $A$ such that $\varphi_x(a)=x$. We can take $\varphi_x$ to be a pure state extension of the point-evaluation at $x$. Then the cyclic vector is an $x$-eigenvector of $\pi_x(a)$, where $\pi_x$ is the representation associated to $\varphi_x$. Since the eigenvectors corresponding to distinct eigenvalues are orthogonal, in every irreducible representation $\pi$ of $A$ the operator $\pi(a)$ has only countably many eigenvectors. By a counting argument, $A$ has at least $2^{\aleph_0}$ spatial equivalence classes of irreducible representations. 
\end{proof}

\section{Concluding Remarks} 
 
Our title was inspired by the title of the ground-breaking Shelah's paper \cite{Sh:Can}, but the answers to the questions posed in these titles are quite different. Solovay's inaccessible may or may not be taken away depending on whether one requires the Baire-measurability alone, or the Lebesgue-measurability as well. In our case, the $\diamondsuit_{\aleph_1}$ is not necessary for the construction. The question whether a counterexample to Naimark's problem can be constructed in $\ax{ZFC}$ alone, in $\ax{ZFC}+\ax{CH}$, or using $\diamondsuit_\kappa$ for some $\kappa\geq\aleph_2$, remains open.

Around 2010, the senior author conjectured that Naimark's problem has positive answer in a model obtained by adding a sufficient number (supercompact cardinal, if need be) of Cohen reals. Theorem \ref{main} and its proof give some (inconclusive) support for the negation of this conjecture. Additional support would be provided by a proof that a forcing notion with the properties of $\bE_{A^\circ}(\bar\varphi,\bar\psi)$ and the countable chain condition, can be defined for tuples of inequivalent pure states for every simple and unital (not necessarily separable) \cstar-algebra. The experience suggests that the countable chain condition and non-commutativity do not mix well:  Note the increase of complexity between  \cite[\S 3]{farahembedding} and \cite[\S 4]{farahembedding}, as well as the proof that a natural noncommutative generalization of the poset for adding many Cohen reals does not have the countable chain condition given in   \cite[Lemma~4.1]{Fa:AllAll}. Moreover, if $A$ has irreducible representations on both separable and non-separable Hilbert spaces (see e.g., \cite[Theorem~10.4.3]{Fa:STCstar}), then adding an automorphism of $A$ that moves one of the associated pure states to another, necessarily collapses $\aleph_1$. Thus, the relevant question is whether such forcing can be constructed for \cstar-algebras that are inductive limits that appear in the course of the proof of Theorem \ref{main}. 

Another possible route towards constructing a counterexample to Naimark's problem would be the following: Instead of forcing with $\bEA{\bar\varphi}{\bar\psi}$ or a modification thereof, find a separable \cstar-subalgebra $B$ of $A$ such that the restrictions $\bar\varphi'$ and $\bar\psi'$ to $B$ of all the tuples of pure states involved, uniquely determine their extensions to $A$. Force with $\bEB{\bar\varphi'}{\bar\psi'}$ to produce a generic automorphism $\Phi$ of $B$ such that $\bar\varphi'$ and $\bar\psi'$ have unique, and equivalent, extensions to $B\rtimes_{\Phi} \bbZ$. This plan hinges on the answer to the following purely \cstar-algebraic question.\footnote{Note, however, that we may assume $B$ is an elementary submodel of $A$; see \cite[Appendix~D]{Fa:STCstar}.} For simplicity, it is stated for single pure states instead of $m$-tuples. 

\begin{quest}
Suppose that $A$ is a unital \cstar-algebra, $B$ is a unital \cstar-subalgebra of $A$, both $A$ and $B$ are simple, and $\varphi'$ and $\psi'$ are pure states of $B$ with the unique pure state extensions to $A$. In addition, suppose that $\Phi$ is a sufficiently generic automorphism of $B$ such that $\varphi'$ and $\psi'$ have unique and unitarily equivalent pure state extensions to $B\rtimes_\Phi \bbZ$. Is there an amalgamation $C$ of $B\rtimes_\Phi\bZ$ and $A$ such that $\varphi$ and $\psi$ have unique pure state extensions to $C$?
\end{quest}

Such amalgamation would be a `partial crossed product' of sorts of $A$ by an automorphism of $B$. There is a rich literature on partial crossed products (see \cite{exel2017partial} and the references thereof), but our situation does not satisfy the requirements imposed on partial dynamical systems in \cite[Definition~6.4]{exel2017partial}. 

\begin{quest}\label{muyouter}
Let $A$ be a separable, simple, non-type I \cstar-algebra. Does there exist an automorphism $\Theta$ of $A$ such that $\sigma\circ\Theta\nsim\sigma$ for every pure state $\sigma$ of $A$?
\end{quest}

If a \cstar-algebra $A$ has an automorphism $\Theta$ as in Question \ref{muyouter}, then for every $m\geq 1$ and tuples $\bar\varphi$ in $\bar\psi$ in $\pure_m(A)$, it has an automorphism $\Phi$ with the same property that in addition satisfies $\bar\varphi\circ \Phi=\bar\psi$. Such an automorphism can be obtained by conjugating $\Theta$ by a Kishimoto--Ozawa--Sakai-type automorphism as in Theorem \ref{preglimm}. Assuming that in addition one could assure that all pure states of $A$ have unique pure state extensions to a crossed product associated with $\Theta$, one would secure the assumptions of the following. 

\begin{prop}\label{P.scifi}
Suppose that there exists a class $\cA$ of separable, simple, unital \cstar-algebras such that:
\begin{enumerate}
    \item\label{2.P.scifi} $\cA$ is closed under inductive limits, and
    \item\label{3.P.scifi} For every $A\in \cA$ and pure states $\varphi$ and $\psi$ of $A$, there exists an extension $B\in\cA$ of $A$ such that (i) $\varphi$ and $\psi$ have equivalent pure state extensions to $B$ and (ii) every pure state of $A$ has a unique pure state extension to $B$. 
\end{enumerate}
Then $\ax{CH}$ implies that there is a counterexample to Naimark's problem.   
\end{prop}

\begin{proof}
Suppose that $\ax{CH}$ holds, and fix $X\subseteq\aleph_1$ such that the inner model $L[X]$ (see \cite[Definition~II.6.29]{Ku:Set}) contains all reals. Then $\diamondsuit_{\aleph_1}$ holds in $L[X]$ (see \cite[Exercise~III.7.21]{Ku:Set}). Working in $L[X]$, modify the construction of a counterexample as in \cite{AkeWe:Consistency} (see also \cite[Theorem~11.2.2]{Fa:STCstar}) as follows: One constructs an inductive system of \cstar-algebras $A_\alpha$, for $\alpha<\aleph_1$, in $\cA$ so that at every successor step of the construction the extension $A_{\alpha+1}$ of $A_\alpha$ is chosen using $\diamondsuit_{\aleph_1}$ and \eqref{3.P.scifi}. At limit stages, take inductive limits. The inductive limit $A$ of this system is a counterexample to Naimark's problem in $L[X]$, by a proof analogous to those in \cite{AkeWe:Consistency} or \cite[Theorem~11.2.2]{Fa:STCstar}.  

We claim that $A$ remains a counterexample to Naimark's problem in the universe $V$. Assume otherwise. Since it is a counterexample to Naimark's problem in $L[X]$, there exists a pure state $\eta$ of $A$ that belongs to $V$ but not to $L[X]$. The set
\begin{equation*}
    \sfC:=\{\alpha<\aleph_1:\text{ the restriction of }\eta\text{ to }A_\alpha\text{ is pure}\}   
\end{equation*}
includes a club (see \cite[Proposition~7.3.10]{Fa:STCstar}). Let $\alpha:=\min(\sfC)$. By induction on countable ordinals $\beta\geq \alpha$, one proves that $\eta\rs A_\beta$ is the unique pure state extension of $\eta\rs A_\alpha$ to $A_\beta$ in $L[X]$, for every $\beta<\aleph_1$. At the successor stages this is a consequence of the choice of $A_{\beta+1}$, and at the limit stages it is automatic. This provides a definition of $\eta$ in $L[X]$; contradiction. 
\end{proof}

The proof of Proposition \ref{P.scifi} begs the question: Is it possible to add a new pure state to a counterexample to Naimark's without adding new reals? The answer is, at least assuming $\diamondsuit_{\aleph_1}$, positive (see \cite[Exercise~11.4.11]{Fa:STCstar}); one can see that $\diacoh+\ax{CH}$ suffices for this construction). 

The main result of \cite{farah2016simple} is a construction (under the assumption of $\diamondsuit_{\aleph_1}$) of a nuclear, simple \cstar-algebra not isomorphic to its opposite algebra. We do not know whether the existence of an algebra with this property follows from $\diacoh+\ax{CH}$. In the same theorem, a counterexample to Glimm's dichotomy with exactly $\aleph_0$ inequivalent pure states was constructed using $\diamondsuit_{\aleph_1}$. Such construction using $\diacoh$ would require a generalization of the forcing $\bEA{\bar\varphi}{\bar\psi}$ to countable sequences of inequivalent pure states.  

As pointed out in \cite[Section 7]{farah2021nonseparable}, there is a possibility that the methods developed there may give a $\ax{ZFC}$ construction of a simple \cstar-algebra not isomorphic to its opposite. These methods are however unlikely to be relevant to Naimark’s problem (see the discussion of \cite{suri2017naimark} below).

In \cite{suri2017naimark}, it was shown that a counterexample to Naimark's problem cannot be a graph \cstar-algebra (not to be confused with the `graph CCR algebras' of \cite[\S 10]{Fa:STCstar}). We conjecture that a sweeping generalization of this result holds: If a \cstar-algebra $A(\Gamma)$ is defined from a discrete object (graph, group, groupoid, semigroup, etc.) $\Gamma$ in a way that assures that (using the notation of Appendix \ref{S.generic}) $A(\Gamma)$ as computed in $M$ is dense in $A(\Gamma)$ as computed in $M[\kG]$ for all $M$ and all $M$-generic filter $\kG$, then $A(\Gamma)$ is (provably in $\ax{ZFC}$) not a counterexample to Naimark's problem.   

By \cite[Corollary~2.3]{kishimoto1981outer}, an automorphism $\Phi$ of a separable and simple \cstar-algebra satisfies $\varphi\circ\Phi\sim\varphi$ for all pure states $\varphi$ of $A$ if, and only if, it is inner. Theorem \ref{agap} and Theorem \ref{preglimm} imply that both the generic automorphism $\Phi_{\kG}$ and the conjugate of a ground-model outer automorphism by $\Phi_{\kG}$ send every ground-model pure state to an inequivalent pure state. We conjecture that this property is shared by every reduced word in outer automorphisms of $A$ and $\Phi_{\kG}$ in which the latter occurs. This resembles the properties of generic automorphisms (and anti-automorphisms) of II$_1$ factors as exhibited in \cite[Lemma A.2]{Ioana-Peterson-Popa} and \cite{vaes2009factors}, and used there to construct interesting examples of II$_1$ factors with a separable predual. These lemmas, combined with an iterated crossed product construction \`a la Akemann--Weaver propelled by $\diamondsuit_{\aleph_1}$ was used in \cite{farah2021rigid} to construct a hyperfinite II$_1$ factor with non-separable predual and not isomorphic to its opposite. It is not difficult to see that $\diacoh+\ax{CH}$ in place of $\diamondsuit_{\aleph_1}$ suffices for this construction.

\appendix

\section{\cstar-algebras in generic extensions}\label{S.generic}

In this appendix, we state and prove a few straightforward results on \cstar-algebras in models of \ZFCmP. In particular, we will prove results about the relation between $A^M$ and $A^{M[\kG]}$ that we could not find in the literature.

Recall that a property \textbf{P} of \cstar-algebras is said to be \emph{absolute} if for every $A$, and every $M$ and $M[\kG]$ as in \S\ref{S.models} and \S\ref{S.forcing}, $A^M$ has \textbf{P} if, and only if, $A^{M[\kG]}$ has \textbf{P}.\footnote{Purists may prefer the term `forcing absolute' but the difference can be ignored in the context of this paper.}

\begin{lem}\label{absolute}
Both being simple and being non-type I are absolute properties of \cstar-algebras. 
\end{lem}

\begin{proof}
To prove that simplicity is absolute, we first consider the case when $A$ is separable and unital. Fix a countable norm-dense subset $D$ of $A$. Then some $x\in A$ generates a proper two-sided ideal if, and only if, for every $m\in\bbN$ and all $m$-tuples $(a_j:j<m)$ and $(b_j:j<m)$ of elements of $D$, we have that $\|1_A-\sum_{i<m} a_ixb_i\|\geq 1$. This is because $\|1-a\|<1$ implies that $a$ is invertible (see \cite[Lemma~1.2.6]{Fa:STCstar}). Thus, the assertion `$A$ has a proper two-sided ideal' is a $\mathbf\Sigma^1_1$-statement (with some code for $A$ as a parameter) and is therefore absolute between all transitive models of \ZFCmP (see \cite[Theorem 25.20]{Jech:SetTheory} or \cite[Theorem~B.2.11]{Fa:STCstar}). 

If $A$ is not necessarily unital, then a similar argument show that the assertion `some $x\in A$ generates a proper ideal of $A$' is $\mathbf\Sigma^1_2$, and therefore absolute between all transitive models of \ZFCmP{} that contain all countable ordinals (see \cite[Theorem 25.20]{Jech:SetTheory} or \cite[Theorem~B.2.11]{Fa:STCstar}). 

If $A$ is not necessarily separable, then standard reflection arguments show that it is simple if, and only if, it is an inductive limit of a directed family of its separable, simple \cstar-subalgebras (see \cite[\S 7.3]{Fa:STCstar}), and the conclusion follows.

We can now prove the absoluteness of being non-type I. Since the completion in a forcing extension of the ground-model CAR algebra is isomorphic to the CAR algebra as calculated in the forcing extension, the upwards absoluteness of being non-type I follows by Glimm's theorem (see \cite[Theorem~3.7.2]{Fa:STCstar}). 
\end{proof}

It is worth mentioning that every axiomatizable (in logic of metric structures, see \cite{Muenster}) property of \cstar-algebras is absolute. The reason for this is that the ball of radius $n$ in $A^M$ is dense in the ball of radius $n$ in $A^{M[\kG]}$ for all $n$, and therefore the suprema and infima of continuous functions on these two sets agree. More generally, properties definable by uniform families of formulas (see \cite[Definition~5.7.1.1]{Muenster}) are absolute. Together with \cite[Theorem~2.5.0.1 and  Theorem~5.7.1.3]{Muenster}, this implies that many important properties of \cstar-algebras are absolute. 

An example of a non-absolute property is separability. Also, being isomorphic to $\sB(H)$ or to the Calkin algebra is not absolute, since (for example) adding new subsets of $\bbN$ adds new projections to the atomic masa that are not in the norm-closure of the ground-model projections. This shows that $\cB(H)^M$ is not dense in $\cB(H)^{M[\kG]}$; in order to assure that they are not even isomorphic, one can for example increase the cardinality of the continuum. 

\begin{cor}\label{bacan}
If $A$ is a separable, simple, unital, and non-type I \cstar-algebra, then $\bE_{A^\circ}(\bar\varphi,\bar\psi)$ forces that $A\rtimes_{\Phi_\kG}\bZ$ has all of these properties. 
\end{cor}

\begin{proof}
By Lemma \ref{absolute}, these properties of $A$ are absolute. The crossed product is therefore unital and non-type I, and it remains to prove that it is simple. Being non-type I, $A$ has continuum many inequivalent pure states and Theorem \ref{agap} implies that $\Phi_\kG$ moves a pure state to an inequivalent pure state, and is therefore outer. By \cite[Theorem~3.1]{kishimoto1981outer}, this implies that the crossed product is simple. 
\end{proof}

\begin{lem} \label{L.absolute}
For a \cstar-algebra $A$, and models $M$ and $M[\kG]$ as described in \S\ref{S.forcing}, we have the following: 
\begin{enumerate}
\item\label{1.absolute} $\U(A^M)$ is norm-dense in $\U(A^{M[\kG]})$.	
\item\label{2.absolute} If $\varphi$ and $\psi$ are pure states of $A$ in $M$, then their (unique) pure state extensions to $A^{M[\kG]}$ are equivalent in $M[\kG]$ if, and only if, $\varphi$ and $\psi$ are equivalent in $M$.   
\end{enumerate}
\end{lem}

\begin{proof}
Note that \eqref{1.absolute} is a consequence of the continuous functional calculus, as follows: Suppose that $(a_n:n\in\bN)$ is a sequence of elements of $A$ in $M[\kG]$ that converges to a unitary $u$. Then $\|a_n^* a_n -1\|\to 0$ and $\|a_n a_n ^* -1\|\to 0$ as $n\to \infty$. Therefore, $a_n^* a_n$ is invertible for a large enough $n$, and hence $|a_n|:=(a_n^*a_n)^{1/2}$ is invertible for a large enough $n$. The unitary from the polar decomposition of $a_n$, $u_n:= a_n|a_n|^{-1}$,  satisfies $\|u_n-u\|\to 0$.

To see \eqref{2.absolute}, as pointed out in \S\ref{S.equiv}, $\varphi$ and $\psi$ are equivalent if, and only if, there is a unitary $u$ such that $\|\varphi\circ \Ad u-\psi\|<2$, hence the conclusion follows from \eqref{1.absolute}. 
\end{proof}

Every definable (in the model-theoretic sense, see \cite[\S 3]{Muenster}) subset of $A$ has the absoluteness property proved for $\U(A)$ in Lemma \ref{L.absolute} by a proof analogous to that of Lemma \ref{L.absolute}.  

The following was essentially proved in \cite[Proposition 6]{AkeWe:Consistency}. We include its proof for completeness. 

\begin{prop}\label{newpure}
Let $M$ be a countable transitive model of \ZFCmP, let $\bP$ be a forcing notion in $M$, and let $A\in M$ be a unital, non-type I \cstar-algebra. If $\bP$ adds a new real to $M$, then it adds a new pure state to $A$ which is inequivalent to any ground-model pure state of $A$.
\end{prop}

\begin{proof}
The construction is very similar to the one in the proof of Theorem 5.5.4 in \cite{Fa:STCstar}, where additional details can be found. For $i<2$ define a linear functional $\delta_i$ on $M_2(\bbC)$ by
\begin{equation*}
    \delta_i\left(
    \begin{pmatrix}
    \lambda_{00} & \lambda_{01}\\
    \lambda_{10} & \lambda_{11}
    \end{pmatrix}
    \right)=\lambda_{ii}.
\end{equation*}
This is a pure state of $M_2(\bbC)$. For $r\in 2^\bbN$ define a linear functional $\varphi_r$ on the CAR algebra as follows: on the elementary tensors (note that in $\bigotimes_{n\in \bN} a_n$, we have that $a_n=1$ for all but finitely many $n$) let
\begin{equation*}
    \varphi_r\left(\bigotimes_{n\in\bN}a_n\right)=\prod_{n\in\bN}\delta_{r(n)}(a_n).
\end{equation*}
The linear extension of $\varphi_r$ (still denoted $\varphi_r$) is a pure state on $\CAR$. By Glimm's theorem (see \cite[Theorem~3.7.2]{Fa:STCstar}), $A$ includes some separable \cstar-subalgebra $B$ which has the CAR algebra as a quotient. The composition of $\varphi_r$ with the quotient map is a pure state of $B$, and this pure state can be extended to a pure state $\psi_r$ of $A$. Clearly, $r$ can be recovered from $\psi_r$ by evaluation.

Suppose that $\psi_r$ is equivalent to a ground-model pure state. Since $\U(A)^M$ is norm-dense in $\U(A)^{M[\kG]}$, there exists a ground-model pure state $\sigma$ of $A$ such that $\|\psi_r-\sigma\|<1$. Then the restriction of $\sigma$ to $B$ still factors through the quotient map to a state of the CAR algebra and $r$ can be recovered from this state. But this implies that $r$ belongs to the ground model; contradiction.  
\end{proof}

Suppose $M$ is a transitive model of \ZFCmP{} and $X$ is a Polish space with a code in $M$. By a result of Solovay, an element $r$ of $X$ is Cohen-generic over $M$ if, and only if, it belongs to every dense open subset of $X$ coded in $M$ (see \cite[Lemma 26.24]{Jech:SetTheory}). Thus, there exists a Cohen-generic element of $X$ over $M$ if, and only if, the closed nowhere dense subsets of $X$ with codes in $M$ do not cover $X$.

The minimal cardinality of a family of nowhere dense sets that cover the real line is denoted by $\cov(\cM)$ (see \cite[\S 8.4]{Fa:STCstar}). The Baire category theorem implies that this cardinal is uncountable, and Martin's Axiom for $\kappa$ dense sets implies that $\cov(\cM)>\kappa$ (see \cite{Ku:Set}). 

\begin{cor}
Let $A$ be a separable, simple, unital and non-type I \cstar-algebra, and let $X\subseteq\pure(A)$ with $|X|<\cov(\cM)$. Then for every pair of inequivalent pure states $\varphi$ and $\psi$ on $A$ there exists some $\Phi\in\aut(A)$ such that $\varphi\circ\Phi=\psi$ and every $\rho\in X$ has a unique pure state extension to $A\rtimes_\Phi\bZ$.  
\end{cor}

\begin{proof} 
Let $M$ be an elementary submodel of $H((2^{\aleph_0})^+)$ such that $A\in M$, $X\subseteq M$, and $|M|<\cov(\cM)$. Let $\bar M$ be the transitive collapse of $M$. Then the nowhere dense subsets of $\bbR$ coded in $\bar M$ are too few to cover $\bbR$. By a result of Solovay, a real that does not belong to any of these sets is Cohen-generic over $\bar M$ (see \cite[Lemma 26.24]{Jech:SetTheory}). By Theorem \ref{sumup}, there exists an $\bar M$-generic filter $\kG$ on $\bEA{\bar\varphi}{\bar\psi}$, and by Corollary \ref{uep}, $\Phi:=\Phi_\kG$ is as required.   
\end{proof}

\section{The combinatorial principle $\diacoh$}\label{S.diamond}

This section contains only set-theoretic considerations: we prove that the combinatorial principle $\diacoh$ does not imply Jensen's $\diamondsuit_{\aleph_1}$ and that it does not decide the cardinality of the continuum. 

As the attentive reader may have noticed during the proof of Theorem \ref{realmain} (or Theorem \ref{T.separable}), the combinatorial principle $\diacoh$ can be thought as an oracle in which the required tasks at successor steps can be done by the mean of a Cohen real. On the upside, and in opposition to the usual application of Jensen's $\diamondsuit_{\aleph_1}$, such tasks can be delayed (this is the job of the book-keeping) and they do not have to be handled at the moment they are captured by the oracle. Our weakening of $\diamondsuit_{\aleph_1}$ is, at the end of the day, a sort of guessing-plus-forcing axiom in which the generic objects exist (in a prescribed extension) only for countable forcing notions that are elements of models whose job is to capture subsets of $\aleph_1$ correctly.

\begin{lem}\label{lemon}
It is relatively consistent with $\ax{ZFC}$ that $\diacoh+\ax{CH}+\neg\diamondsuit_{\aleph_1}$ holds.
\end{lem}

\begin{proof} 
Let $M_0$ be a countable, transitive model of a large enough fragment of $\ax{ZFC}+\ax{CH}$ in which $\diamondsuit_{\aleph_1}$ fails. Such model was first constructed by Jensen (see \cite{devlin2006souslin}, also \cite[\S V]{Sh:PIF}). Let $(\bbP_\alpha,\dot\bbQ_\alpha)_{\alpha<\aleph_1}$ be a finite support iteration of non-trivial ccc forcing notions, each of which has cardinality at most $\aleph_1$. Let $\kG\subseteq \bbP_{\aleph_1}$ be an $M_0$-generic filter. By the countable chain condition, $\diamondsuit_{\aleph_1}$ fails in $M_0[\kG]$ (see \cite[Exercise IV.7.57]{Ku:Set}) and the standard `counting of names' argument shows that the Continuum Hypothesis holds in $M_0[\kG]$. 
 
For $\alpha<\aleph_1$, $M_\alpha:=M_0[\kG\cap \bbP_{\omega\cdot\alpha}]$ (here $\omega\cdot\alpha$ is the $\alpha$th limit ordinal) is the intermediate forcing extension. By the countable chain condition again, no reals are added at stages of uncountable cofinality (see \cite[Lemma 18.9]{LorBook}), and therefore every real in $M_0[\kG]$ belongs to some $M_\alpha$ for $\alpha<\aleph_1$. Since a finite support iteration of non-trivial ccc forcings adds a Cohen real at every limit stage of countable cofinality (see \cite[Exercise V.4.25]{Ku:Set}), for every $\alpha<\aleph_1$ the model $M_{\alpha+1}$ contains a real that is Cohen-generic over $M_\alpha$. 

Fix a name for a subset $X$ of $\aleph_1$. Yet again, by the countable chain condition and the standard closing off argument, there is a club $C\subseteq \aleph_1$ such that for every $\alpha\in C$, the forcing notion $\bbP_\alpha$ adds $X\cap \alpha$. Therefore, $X\cap \alpha\in M_\alpha$ for stationary many $\alpha$, and $\bbP_{\aleph_1}$ forces that $\diacoh$ holds.    
\end{proof}

The following corollary exhibits a substantial difference between the two principles $\diamondsuit_{\aleph_1}$ and $\diacoh$.

\begin{cor}\label{C.CH}
The principle $\diacoh$ does not decide the value of $\kc$. 
\end{cor}

\begin{proof}
If in the proof of Lemma \ref{lemon} we begin with a model of $2^{\aleph_0}=\kappa$, then $M_0[\kG]$ is a model of $\diacoh+2^{\aleph_0}=\kappa$.
\end{proof}

To see that $\diacoh$ is not a consequence of $\ax{CH}$, we will show that, unlike the Continuum Hypothesis, $\diacoh$ implies the existence of a Suslin tree.

In \cite{moore}, Moore, Hru\v{s}\'ak, and D\v{z}amonja introduced a variety of parametrized $\diamondsuit_{\aleph_1}$ principles based on the weak diamond (see \cite{devlin1978weak}) which have a similar relationship to $\diamondsuit_{\aleph_1}$ as cardinal invariants of the continuum have to $\ax{CH}$.

\begin{defin}[\cite{moore}]
The principle $\diamondsuit(\ax{non}(\sM))$ holds if for every function
\begin{equation*}
    F\colon2^{<\aleph_1}\to\sM 
\end{equation*}
such that $F\upharpoonright2^\alpha$, for $\alpha<\aleph_1$, is Borel, there exists some $g\colon\aleph_1\to\bR$ such that for all $f\colon\aleph_1\to 2$, the set $\{\alpha<\aleph_1:g(\alpha)\notin F(f\upharpoonright\alpha)\}$ is stationary.
\end{defin}

\begin{prop}\label{moore}
The principle $\diamondsuit(\ax{non}(\sM))$ is a consequence of $\diacoh$.
\end{prop}

\begin{proof}
Let $(M_\alpha:\alpha<\aleph_1)$ be a $\diacoh$-chain, and let $F\colon2^{<\aleph_1}\to\sM$ be such that for all $\alpha<\aleph_1$ the restriction $F\upharpoonright2^\alpha$ is Borel. For each $\alpha<\aleph_1$, let $r_\alpha\in\bN^\bN$ be such that $F\upharpoonright2^\alpha$ is definable from $r_\alpha$ and let $\alpha\leq\phi(\alpha)<\aleph_1$ be such that $r_\alpha\in M_{\phi(\alpha)}$. Define $g\colon\aleph_1\to\bR$ by choosing $g(\alpha)$ to be Cohen-generic over $M_{\phi(\alpha)}$. Let $f\colon\aleph_1\to 2$ be arbitrary. Since $\{\alpha<\aleph_1:f\upharpoonright\alpha\in M_\alpha\}$ is stationary, and $g(\alpha)$ is Cohen-generic over a model containing both $f\upharpoonright\alpha$ and $r_\alpha$, then $\{\alpha<\aleph_1:g(\alpha)\notin F(f\upharpoonright\alpha)\}$ is stationary as well.
\end{proof}

\begin{cor}
Following the notation above.
\begin{enumerate}
    \item If $\diacoh$ holds then there is a Suslin tree.
    \item The principle $\diacoh$ is not a consequence of $\ax{CH}$.
\end{enumerate}
\end{cor}

\begin{proof}
By \cite[Theorem 3.1]{moore}, $\diamondsuit(\ax{non}(\sM))$ implies that there is a Suslin tree and therefore (1) follows from Proposition \ref{moore}. (2) follows from (1) and the fact that $\ax{CH}$ does not imply the existence of a Suslin tree (\cite{devlin2006souslin}). 
\end{proof}

One could consider $\diamondsuit^{\ax{Random}}$, $\diamondsuit^{\ax{Hechler}}$, or diamonds associated to other Suslin ccc forcings. The countable chain condition of the forcing is used in order to assure the property $\diacoh$\eqref{2.diacoh} in Definition \ref{Def.diacoh}. We are not aware of any applications of these axioms.

\bibliographystyle{amsplain}
\bibliography{naimark}
\end{document}